\documentclass[12pt,a4paper,reqno]{amsart}
\usepackage[T1]{fontenc}
\usepackage{lmodern}
\usepackage{amsmath,amssymb,mathtools,mathrsfs}
\usepackage{microtype}
\usepackage[a4paper,hmargin=28mm,top=26mm,bottom=28mm,includeheadfoot]{geometry}
\usepackage{tikz-cd}
\usepackage{booktabs}
\usepackage{xurl}
\usepackage{xcolor}
\definecolor{linkblue}{RGB}{35,64,105}
\usepackage[colorlinks=true,linkcolor=linkblue,citecolor=linkblue,urlcolor=linkblue]{hyperref}

\numberwithin{equation}{section}
\newtheorem{theorem}{Theorem}[section]
\newtheorem{lemma}[theorem]{Lemma}
\newtheorem{proposition}[theorem]{Proposition}
\newtheorem{corollary}[theorem]{Corollary}
\theoremstyle{definition}

\newtheorem{remark}[theorem]{Remark}
\newtheorem*{remark*}{Remark}

\newcommand{\X}{\mathcal X}
\newcommand{\ddc}{dd^c}
\newcommand{\R}{\mathbb R}
\newcommand{\Q}{\mathbb Q}
\newcommand{\CC}{\mathbb C}
\newcommand{\OO}{\mathcal O}
\DeclareMathOperator{\PSH}{PSH}
\DeclareMathOperator{\vol}{vol}
\DeclareMathOperator{\Supp}{Supp}

\DeclareMathOperator{\MA}{MA}
\title[Invariance of logarithmic plurigenera and adjoint volumes]
{Invariance of logarithmic plurigenera and adjoint volumes in K\"ahler families}

\author{Kefeng Liu}
\address{Mathematical Sciences Research Center, Chongqing University of Technology,
Chongqing 400054, China}
\email{kefengliu@cqut.edu.cn}

\author{Xueyuan Wan}
\address{Mathematical Sciences Research Center, Chongqing University of Technology,
Chongqing 400054, China}
\email{xwan@cqut.edu.cn}
\makeatletter
\@namedef{subjclassname@2020}{\textup{2020} Mathematics Subject Classification}
\makeatother
\hypersetup{pdftitle={Invariance of logarithmic plurigenera and adjoint volumes in Kahler families},
pdfauthor={Kefeng Liu and Xueyuan Wan}}

\begin{document}

\begin{abstract}
We consider smooth proper families of compact complex manifolds over
a disk whose total space is K\"ahler, together with a boundary divisor
with relative simple normal crossings. We prove that, if the boundary
is a rational divisor with coefficients in $[0,1)$, every logarithmic
pluricanonical form on a fiber extends to a neighborhood of that
fiber, and hence that the logarithmic plurigenera are constant. No
positivity of the log canonical class is assumed. We prove further
that, for real boundary coefficients in $[0,1]$, a log-adjoint class
twisted by a smooth semipositive form has the same volume on every
fiber, and that it is pseudoeffective either on every fiber or on
none. When the log-adjoint class is nef on every fiber, we give a
second proof of the extension theorem, which does not use the minimal
model program.
\end{abstract}

\maketitle

\section{Introduction}

The plurigenera $h^0(X,\OO_X(mK_X))$ of a compact complex manifold
are among the basic invariants of its birational geometry. For a
pair $(X,B)$, where $B=\sum_ib_iD_i$ is an effective rational
divisor with distinct prime components $D_i$, the logarithmic
plurigenera are defined by
\[
 P_m(X,B)=h^0\bigl(X,\OO_X(m(K_X+B))\bigr),
 \qquad m\ge1,\ mB\ \text{integral};
\]
a section of $\OO_X(m(K_X+B))$ is an $m$-canonical form with poles of
order at most $mb_i$ along $D_i$. In a holomorphic family these
dimensions are upper semicontinuous, so that they might jump up at a
special fiber. In order to exclude this, it suffices to show that
every section on the special fiber extends to a neighborhood of that
fiber. In this paper we consider this extension problem for families
of pairs with relative simple normal crossings whose total space is
K\"ahler. We shall prove that the logarithmic plurigenera of such
families are invariant, and that the volumes of their log-adjoint
classes, which are real cohomology classes, are invariant as well.

For projective families the invariance of plurigenera was proved by
Siu, first for families of general type \cite{Siu98} and then in
general \cite{Siu02}. Siu constructed singular metrics adapted to
pluricanonical sections and applied the $L^2$ extension theorem of
Ohsawa--Takegoshi \cite{OT}. Kawamata and Takayama obtained
deformation and extension results for varieties with canonical
singularities \cite{Kawamata,Takayama}, and P\u{a}un \cite{Paun07}
simplified the construction of the metrics by a one-tower argument
and proved a twisted $L^2$ extension theorem. For projective
log-smooth Kawamata log terminal (klt) pairs, Hacon--McKernan--Xu
proved the invariance under the assumption that either the log
canonical divisor or the boundary is relatively big
\cite[Theorem~1.8]{HMX}. Later they removed this assumption
\cite[Theorem~4.2]{HMX18}, using the quantitative extension theorem
of Berndtsson--P\u{a}un \cite[Theorem~0.2]{BP}.

For K\"ahler families, Levine obtained early results on
pluricanonical divisors under deformations \cite{Levine83,Levine85},
and Siu formulated the general invariance problem in \cite{Siu02}.
Cao--P\u{a}un proved infinitesimal extension results under
integrability conditions \cite{CP}. Related positivity theorems for
relative canonical and adjoint classes were proved by Schumacher
\cite{Schumacher}, P\u{a}un \cite{Paun17}, Guenancia \cite{Gue} and
Cao--Guenancia--P\u{a}un \cite{CGP}. Recently, Deng--Ning--Wang--Zhou
proved the invariance for families with nef canonical bundles
\cite{DNWZ}, and Chen--Rao--Wang treated, among other cases, smooth
families having one fiber with nef canonical bundle \cite{CRW}.
Li--P\u{a}un--Wang--Zhou proved the invariance of all plurigenera
for smooth proper families with K\"ahler total space, without any
assumption of nefness \cite{LPWZ}.

In this paper we prove the logarithmic extension theorem for rational
boundaries with relative simple normal crossings and coefficients in
$[0,1)$, again without any assumption of positivity on the log
canonical class. Throughout the paper, $p\colon\X\to\Delta$ denotes a
proper holomorphic submersion onto a disk $\Delta\subset\CC$ centered
at $0$, with coordinate $t$. Its fibers $X_t=p^{-1}(t)$ are connected
compact complex manifolds of dimension $n$, and the total space $\X$
is K\"ahler. We put $K_{\X/\Delta}=K_\X\otimes p^*K_\Delta^{-1}$. For
a bundle, a divisor, a form or a function on $\X$, the subscript $t$
denotes its restriction to $X_t$; for instance, $B_t=B|_{X_t}$. We
use the same convention for the other families below. We say that an
effective divisor $B=\sum_i b_iD_i$ on $\X$ has \emph{relative simple
normal crossings}, or is \emph{relative SNC}, if near every point of
$\X$ there are holomorphic coordinates $(z_1,\ldots,z_n,t)$, with
$p=t$, in which the distinct prime components of its support meeting
the chart are given by $z_1=0,\ldots,z_k=0$, respectively, for some
$0\le k\le n$. Then each prime component is smooth, and all
intersections of distinct components are smooth over the base.

\begin{theorem}\label{thm:main}
Let $p\colon\X\to\Delta$ and $B=\sum_i b_iD_i$ be as above, with
$b_i\in\Q$ and $0\le b_i<1$, and let $m\ge1$ be an integer such that
$mB$ is integral. Then every section of
$
 \OO_{X_0}\bigl(m(K_{X_0}+B_0)\bigr)
$
extends to a section of
$\OO_\X\bigl(m(K_{\X/\Delta}+B)\bigr)$ over $p^{-1}(U)$,
for a sufficiently small disk $U$ about $0$.
Moreover,
\[
 t\longmapsto h^0\bigl(X_t,m(K_{X_t}+B_t)\bigr)
\]
is locally constant, and
$p_*\OO_\X\bigl(m(K_{\X/\Delta}+B)\bigr)$ is locally free
and commutes with holomorphic base change.
\end{theorem}

We say that the direct image commutes with base change if, for every
holomorphic map $\lambda:S'\to\Delta$ of complex spaces, the natural
map from $\lambda^*p_*\OO_\X\bigl(m(K_{\X/\Delta}+B)\bigr)$ to the
direct image of the pulled-back line bundle on $\X\times_\Delta S'$
is an isomorphism; see also the discussion after
Theorem~\ref{thm:smooth-base}.

Theorem~\ref{thm:main} extends the logarithmic invariance theorem
\cite[Theorem~4.2]{HMX18} from projective families to K\"ahler
families. For $B=0$ it yields the local extension and invariance
statements of \cite[Theorem~1.2]{LPWZ}. For families with K\"ahler
total space it also gives the conclusion of \cite[Theorem~1.6]{FLR}
without the assumption that $X_0$ is projective or that $(X_0,B_0)$
is canonical. Moreover, we do not assume bigness, and the boundary
may have prime components in common with the divisorial negative
part. In particular, the log canonical classes $K_{X_t}+B_t$ need not
be nef, and their volumes may be zero.

For a real divisor $D$ we denote by $\lfloor D\rfloor$ its
coefficientwise round-down. For each $m$ the rational boundary
$m^{-1}\lfloor mB\rfloor$ again has coefficients in $[0,1)$, and,
since $B$ is relative SNC, rounding commutes with restriction to the
fibers. Applying Theorem~\ref{thm:main} to this boundary, we obtain
the following corollary. Its projective case is also contained in
\cite[Theorem~4.2]{HMX18}.

\begin{corollary}\label{cor:round}
Let $p\colon\X\to\Delta$ be as above, and let $B=\sum_i b_iD_i$
be an effective real divisor with relative simple normal crossings
and $0\le b_i<1$. Then, for every integer $m\ge1$, each section of
$
 \OO_{X_0}\bigl(mK_{X_0}+\lfloor mB_0\rfloor\bigr)
$
extends locally to
$\OO_\X\bigl(mK_{\X/\Delta}+\lfloor mB\rfloor\bigr)$.
The corresponding fiber dimensions are locally constant, and the
direct image is locally free and commutes with holomorphic base change.
\end{corollary}

Our second main result concerns the positivity and the volume of
log-adjoint classes in real cohomology. In this result the boundary
coefficients may be real, and the semipositive twist need not be the
first Chern class of a rational line bundle. We first recall some
definitions. Let $Y$ be a complex manifold. The real Bott--Chern
cohomology $H^{1,1}_{\mathrm{BC}}(Y,\R)$ is the space of $d$-closed
real $(1,1)$-forms modulo the forms $\ddc f$ with $f$ smooth and
real. We denote by $[\xi]$ the class of such a form $\xi$, and we
extend $c_1$ linearly to real divisors. Now let $Y$ be a compact
K\"ahler manifold of dimension $n$, and let
$\alpha\in H^{1,1}_{\mathrm{BC}}(Y,\R)$. We say that $\alpha$ is
\emph{K\"ahler} if it contains a K\"ahler form, \emph{nef} if it lies
in the closure of the cone of K\"ahler classes, \emph{pseudoeffective}
if it contains a closed positive $(1,1)$-current, and \emph{big} if it
contains a \emph{K\"ahler current}, that is, a closed $(1,1)$-current
$T$ with $T\ge\omega'$ for some K\"ahler form $\omega'$ on $Y$. The
pseudoeffective classes form a closed convex cone, and the interior
of this cone is the cone of big classes \cite{BouZ}.

Following Boucksom \cite{BouV}, \cite[Section~3.7]{BouZ}, we define
the \emph{volume} of a pseudoeffective class $\alpha$ by
$
 \vol(\alpha)=\sup_T\int_Y T_{\mathrm{ac}}^n,
$
where $T$ ranges over the closed positive $(1,1)$-currents in
$\alpha$ and $T_{\mathrm{ac}}$ denotes the absolutely continuous part
of $T$ in its Lebesgue decomposition. If $\alpha$ is not
pseudoeffective, we put $\vol(\alpha)=0$. A pseudoeffective class is
big if and only if its volume is positive, and the volume of a nef
class $\alpha$ is equal to $\int_Y\alpha^n$ \cite{BouV}. We say that
a line bundle or a real divisor is nef, pseudoeffective or big if its
first Chern class has this property, and we define its volume to be
the volume of its first Chern class.

\begin{theorem}
\label{thm:real-log-volume}
Let $p:\X\to\Delta$ be a proper holomorphic submersion with connected
compact fibers and K\"ahler total space. Let $B=\sum_i b_iD_i$ be an
effective real divisor with relative simple normal crossings and
$0\le b_i\le1$, and let $\beta$ be a smooth closed semipositive real
$(1,1)$-form on $\X$. Put
\[
 \mathfrak a_t=c_1(K_{X_t}+B_t)+[\beta_t].
\]
Then:
\begin{enumerate}
\item if $\mathfrak a_t$ is pseudoeffective for one parameter, it is
      pseudoeffective for every parameter;
\item the function $t\mapsto\vol(\mathfrak a_t)$ is constant.
\end{enumerate}
In particular, if $\mathfrak a_t$ is big for one parameter, then it is
big for every parameter.
\end{theorem}

For $\beta=0$ and projective families, the invariance of the volume
was proved in \cite[Theorem~1.8(3)]{HMX} and
\cite[Corollary~4.3]{HMX18}, even for log canonical boundaries. We
note that Theorem~\ref{thm:real-log-volume} admits coefficients equal
to one. For $B=\beta=0$, the invariance of pseudoeffectivity was
proved by Tsuji for projective families \cite[Theorem~1.2]{Tsuji} and
by Cao--P\u{a}un for K\"ahler families \cite[Corollary~5.5]{CP25};
Hacon--Xie also treated canonical singularities
\cite[Corollary~7.8]{HX}.

For transcendental twists, Hacon--Li--Rao proved the local constancy
of the volume under the assumptions that the distinguished fiber is
projective, that the adjoint class is big and that the moduli class
is relatively nef \cite[Theorem~1.7]{HLR}.
Theorem~\ref{thm:real-log-volume} removes the assumptions of
projectivity and bigness, but requires the stronger condition that
the twist have a smooth semipositive representative on the total
space. Hence it proves \cite[Conjecture~6.6]{HLR} in this case; we do
not treat arbitrary relatively nef twists. It also extends the volume
theorem for threefolds \cite[Theorem~1.8]{HLR} to all dimensions. For
$B=0$, a K\"ahler twist and big adjoint classes, the constancy of the
volume is contained in \cite[Theorem~1.3]{LPWZ}, where the positivity
of the associated K\"ahler--Einstein metrics is also proved.

If we fix the relative divisors $D_i$ and apply
Theorem~\ref{thm:real-log-volume} to each coefficient vector in
$[0,1]^r$, we see that the coefficient vectors for which the adjoint
class is pseudoeffective form a closed convex set independent of the
fiber, and that each coefficient vector gives the same volume on
every fiber (Corollary~\ref{cor:boundary-volume-function}). As for
extension, Theorem~\ref{thm:smooth-twist} admits twists by rational
line bundles with smooth semipositive metrics, and
Theorem~\ref{thm:smooth-base} gives the local freeness of the direct
images over smooth bases and the global extension over Stein bases.
The assumption that the boundary is relative SNC cannot simply be
omitted. In fact, de Fernex and Hacon \cite{dFH} showed that the
logarithmic invariance may fail when the boundary strata degenerate.

\subsection*{Strategy of the proof}
We prove Theorem~\ref{thm:main} along the analytic approach of
\cite{LPWZ}. The extension is obtained from the Ohsawa--Takegoshi
theorem in the form given by Zhou--Zhu \cite{ZZ}. In order to apply
it to a section $s$ of $m(K_{X_0}+B_0)$, we need a positively curved
singular metric on $K_{\X/\Delta}+B$ whose weight on $X_0$ is bounded
from below by $\frac1m\log|s|^2$ up to a constant, where the norm is
taken with respect to a fixed smooth reference metric
(Section~\ref{sec:common-extension}).

We obtain this metric as the limit, as $\varepsilon\downarrow0$, of
normalized solutions of twisted K\"ahler--Einstein equations on the
fibers in the classes $c_1(K_{X_t}+B_t)+\varepsilon[\omega_t]$, where
$\omega$ is a fixed K\"ahler form on $\X$. The densities of these
equations have poles along $B$. We need two facts on a fixed disk:
the solutions vary plurisubharmonically with $t$, and the perturbed
volumes are independent of $t$. Since the volumes are constant, all
the solutions on the fibers can be normalized by one constant on the
total space, and the estimate of Di~Nezza--Guedj--Guenancia
\cite{DGG}, which does not depend on the volume, controls the
normalized potentials even when these volumes tend to zero.

If the log-adjoint class is nef on every fiber, then the perturbed
classes are K\"ahler. Their volumes are then intersection numbers and
hence constant, and the plurisubharmonic variation follows from
Guenancia's theorem on families of conic K\"ahler--Einstein metrics
\cite{Gue}. This gives the short proof in Section~\ref{sec:nef-proof}.

In general the perturbed classes are not nef. In
Section~\ref{sec:relative-log-geometry} we use the K\"ahler minimal
model program and the transcendental base-point-free theorem of
Hacon--Xie \cite{HX} to construct, near each fiber, a model on which
the pullback of the perturbed class is the sum of the pullback of a
K\"ahler class and an effective exceptional divisor. Here the
boundary may have components in common with the divisorial negative
part of the central class. We therefore remove the common part,
construct the relative model, and then show that the removed divisor
is exceptional, so that it can be restored. An open--closed argument
then shows that the perturbed volumes are constant on a fixed disk.
The same argument proves Theorem~\ref{thm:real-log-volume} in the klt
case, and the case of coefficients equal to one follows by
approximation.

On the model the density has zeros along exceptional divisors, in
addition to the poles along the boundary. In
Section~\ref{sec:signed-analysis} we extend the variation argument of
\cite{LPWZ} to such densities. The poles contribute with a favorable
sign, and the error caused by the zeros vanishes, since the
exceptional divisors are orthogonal to the semipositive class. In
Section~\ref{sec:log-descent} we descend the resulting metrics to the
original family and extend them across the discrete set of parameters
excluded in the construction. In Section~\ref{sec:general-proof} we
apply the metric criterion and complete the proof. Applications and
variants for the boundary are given in
Section~\ref{sec:applications}.

\section{A metric criterion for extension}
\label{sec:common-extension}

In this section we reduce the extension problem to the construction
of certain singular metrics. In Lemma~\ref{common:metric-criterion}
we construct, from solutions of twisted K\"ahler--Einstein equations
on the fibers in the perturbed classes, a positively curved metric
whose singularities on the central fiber are controlled by a given
section. In Proposition~\ref{common:weighted-extension} we verify the
integrability condition required in the $L^2$ extension theorem. The
nef case and the general case will both be deduced from these two
results.

\subsection{Conventions}
\label{sec:notation-conventions}

We put
$
 d^c=\frac{\sqrt{-1}}{4\pi}(\bar\partial-\partial),
$ $
 \ddc=\frac{\sqrt{-1}}{2\pi}\partial\bar\partial,
$
so that $\ddc\log|f|^2=[\operatorname{div}(f)]$ for a holomorphic
function $f\not\equiv0$, where $\operatorname{div}$ denotes the
divisor of zeros; we use the same notation for sections of line
bundles. If $(L,h)$ is a Hermitian line bundle and a local frame $e$
of $L$ satisfies $|e|_h^2=e^{-\varphi}$, then the curvature of $h$ is
$\Theta(h)=\Theta(L,h)=\ddc\varphi$. For a divisor $D$ we denote by
$[D]$ the current of integration over $D$, extended linearly to real
coefficients; in cohomological identities $[D]$ also denotes the
class $c_1(D)$.

Let $Y$ be a complex manifold, and let $\xi$ be a smooth closed real
$(1,1)$-form on $Y$. We denote by $\PSH(Y,\xi)$ the set of
$\xi$-psh functions, that is, functions $v\not\equiv-\infty$ which
are locally the sum of a plurisubharmonic (psh) function and a smooth
function and satisfy $\xi+\ddc v\ge0$ in the sense of currents.
Suppose that $Y$ is a compact K\"ahler manifold of dimension $n$ and
that $[\xi]$ is big. Then the Monge--Amp\`ere measure of
$v\in\PSH(Y,\xi)$ is defined to be the nonpluripolar product
\cite{BEGZ}
\[
 \MA_\xi(v)=\left\langle(\xi+\ddc v)^n\right\rangle.
\]
This measure puts no mass on pluripolar sets, that is, on sets which
are locally contained in $\{w=-\infty\}$ for some psh function $w$.
Its total mass is at most $\vol([\xi])$, and we say that $v$ has
\emph{full mass} if equality holds. If $[\xi]$ is pseudoeffective,
the envelope
\[
 V_\xi=
 \left(\sup\{v\in\PSH(Y,\xi):v\le0\}\right)^*,
\]
where ${}^*$ denotes the upper semicontinuous regularization, is
$\xi$-psh and satisfies $\sup_YV_\xi=0$. For every
$u\in\PSH(Y,\xi)$ the function $u-\sup_Yu$ is a competitor in the
definition of $V_\xi$; hence $u\le V_\xi+\sup_Yu$. Thus $V_\xi$ has
minimal singularities in the sense of \cite[Section~1.4]{BEGZ}, and a
function $u\in\PSH(Y,\xi)$ has \emph{minimal singularities} if and
only if $u-V_\xi$ is bounded. We say that a closed positive current
$\xi+\ddc u$ has minimal singularities if $u$ has minimal
singularities; this does not depend on the choice of the smooth
representative $\xi$ of its class. If $[\xi]$ is big, every function
with minimal singularities has full mass \cite{BEGZ}.

\subsection{Reference data}

We allow a twist by a rational line bundle $M$ with a smooth
Hermitian metric $h_M$ of semipositive curvature on the total space;
the case $M=0$ is the setting of Theorem~\ref{thm:main}. By a
rational line bundle we mean a specified rational tensor power of a
holomorphic line bundle, together with the compatible identifications
used in forming its integral powers. We call an integer $m\ge1$
\emph{admissible} if $mB$ and $mM$ are integral. All identities
between line bundles of degree $m$ below are isomorphisms of
holomorphic line bundles, and not only numerical equivalences. We may
assume that $n\ge1$; if the fibers are points, then $p$ is an
isomorphism and there is nothing to prove.

Fix a K\"ahler form $\omega$ on $\X$. We work over a disk whose
closure is contained in the original base. Put
\[
 A=K_{\X/\Delta}+B+M,\qquad
 B=\sum_{i=1}^{r}b_iD_i,\qquad b_i\in\Q,\quad 0\le b_i<1,
\]
where $B$ is relative SNC and $\beta=\Theta(h_M)\ge0$. The volume
forms $\omega_t^n$ on the fibers define a smooth metric $h_K$ on
$K_{\X/\Delta}$. We choose smooth metrics $h_i$ on
$\mathcal O_\X(D_i)$, denote by $\sigma_i$ the canonical section of
$\mathcal O_\X(D_i)$, and put
\begin{equation}\label{common:reference-data}
 h_B=\prod_i h_i^{b_i},\quad h_A=h_Kh_Bh_M,\quad
 \theta=\Theta(h_A),\quad
 \phi_B=\sum_i b_i\log|\sigma_i|_{h_i}^{2},\quad
 F_B=e^{-\phi_B}.
\end{equation}
Then $[\theta_t]=c_1(A_t)$, and $\phi_B$ is the weight of the singular
metric $h_Be^{-\phi_B}$ on $\OO_\X(B)$, whose curvature current is
$[B]$. These rational metrics become genuine metrics in every degree
in which the corresponding line bundles are integral. Multiplying the
$h_i$ by constants, we may assume that $|\sigma_i|_{h_i}\le1$, so
that $F_B\ge1$. Finally we put
\[
 W=\int_{X_t}\omega_t^n,\qquad d\nu_t=W^{-1}\omega_t^n.
\]
Since $W$ is the fiber integral of the closed form $\omega^n$, it is
independent of $t$, and $\nu_t$ is a probability measure.

\subsection{Normalization and the limiting metric}

The normalization and the passage to the limit follow the proof of
\cite[Theorem~8.1]{LPWZ}, where a metric on $K_{\X/\Delta}$ with
fiberwise minimal singularities is constructed. In the following
lemma we adapt this argument to $A$ and to the singular density
$F_B$. The solutions on the fibers, their positivity on the total
space and the constancy of the volumes are assumed. We normalize the
solutions by a constant independent of the parameter. Then the
normalized solutions remain positive on the total space, and we can
pass to the limit even when the perturbed volumes tend to zero.

\begin{lemma}
\label{common:metric-criterion}
Suppose that, for a sequence of positive numbers
$\varepsilon_j\downarrow0$, the following conditions are satisfied on
a fixed disk:
\begin{enumerate}
\item each class $[\theta_t+\varepsilon_j\omega_t]$ is big and
its volume $v_j>0$ is independent of $t$;
\item there is a function
$\phi_j\in\PSH(\X,\theta+\varepsilon_j\omega)$ whose restriction
to every fiber is a full-mass solution of
\begin{equation}\label{common:unnormalized-equation}
 \MA_{\theta_t+\varepsilon_j\omega_t}(\phi_{j,t})
       =e^{\phi_{j,t}}F_{B,t}\omega_t^n.
\end{equation}
\end{enumerate}
Then, for every admissible integer $m$ and every nonzero section
$s\in H^0(X_0,mA_0)$, there exist a constant $C_s$ and a function
$\psi\in\PSH(\X,\theta)$ whose restriction $\psi_0$ to $X_0$
satisfies
\begin{equation}\label{common:section-barrier}
 \psi_0\ge\frac1m\log|s|_{h_A^m}^{2}-C_s.
\end{equation}
If $[\theta_t]$ is pseudoeffective for every $t$, the same
construction gives a function $\psi\in\PSH(\X,\theta)$ whose
restriction to every fiber has minimal singularities. In both
assertions the volumes $v_j$ may tend to zero.
\end{lemma}

\begin{proof}
\emph{Step 1: Normalization.}
Discarding finitely many terms, we may assume that
$0<\varepsilon_j\le1$. Since the reference forms are smooth and the
closed base disk is compact, we have
$\theta_t+\varepsilon_j\omega_t\le C_0\omega_t$ for a constant $C_0$
independent of $j$ and $t$. Hence every
$(\theta_t+\varepsilon_j\omega_t)$-psh function is
$C_0\omega_t$-psh. By the family version of Skoda's integrability
theorem \cite[Theorem~3.4, with Assumption~3.2(2)]{DGG}, there exist
$\alpha>0$ and $H\ge1$ such that
\begin{equation}\label{common:skoda}
 \int_{X_t}e^{-\alpha(v-\sup_{X_t}v)}d\nu_t\le H
 \quad\text{for }v\in\PSH(X_t,\theta_t+\varepsilon_j\omega_t).
\end{equation}
The local triviality required in this assumption refers to local
holomorphic product charts on the total space, which exist since $p$
is a submersion; the fibers need not be biholomorphic to each other.
The constants $\alpha$ and $H$ depend only on the fixed background
family, and not on $v_j$.

Fix $q>1$ such that $qb_i<1$ for all $i$. Then
\begin{equation}\label{common:boundary-Lq}
 \sup_t\|F_{B,t}\|_{L^q(\nu_t)}\le C_F.
\end{equation}
In fact, we cover the family over a slightly smaller closed disk by
finitely many relative SNC charts. In each chart the only unbounded
factor of $F_B^q$ is $\prod_i|z_i|^{-2qb_i}$, and
$\int_0^{r_0}r^{\,1-2qb_i}\,dr<\infty$ for a fixed small coordinate
radius $r_0>0$, while the other factors and the volume densities on
the fibers are uniformly bounded.

Let
\[
 w_{j,t}=\phi_{j,t}-\sup_{X_t}\phi_{j,t},
 \qquad
 J_{j,t}=\int_{X_t}e^{w_{j,t}}F_{B,t}\,d\nu_t.
\]
By Jensen's inequality, \eqref{common:skoda} gives
$\int w_{j,t}\,d\nu_t\ge-\alpha^{-1}\log H$. Applying Jensen's
inequality once more and using $F_B\ge1$, we obtain
$J_{j,t}\ge\int e^{w_{j,t}}\,d\nu_t\ge H^{-1/\alpha}$. Combining this
with $w_{j,t}\le0$ and \eqref{common:boundary-Lq}, we get
\begin{equation}\label{common:J-bounds}
 H^{-1/\alpha}\le J_{j,t}\le C_F,\qquad
 \left\|\tfrac{e^{w_{j,t}}F_{B,t}}{J_{j,t}}\right\|_{L^q(\nu_t)}
       \le H^{1/\alpha}C_F.
\end{equation}
Integrating \eqref{common:unnormalized-equation} over $X_t$, we have
\[
 v_j=e^{\sup_{X_t}\phi_{j,t}}WJ_{j,t}.
\]
Hence the functions
\begin{equation}\label{common:normalized-potentials}
 u_j=\phi_j-\log(v_j/W)
\end{equation}
satisfy
\[
 u_{j,t}=w_{j,t}-\log J_{j,t},\qquad
 v_j^{-1}\MA_{\theta_t+\varepsilon_j\omega_t}(u_{j,t})
       =e^{u_{j,t}}F_{B,t}\,d\nu_t.
\]
Since the constant subtracted in
\eqref{common:normalized-potentials} is independent of $t$, we have
$u_j\in\PSH(\X,\theta+\varepsilon_j\omega)$.

\medskip\noindent
\emph{Step 2: Uniform bounds relative to the envelopes.}
Fix $j$ and $t$. Since $u_{j,t}$ and $w_{j,t}$ differ by a constant,
$w_{j,t}$ has full mass, and
\begin{equation}\label{common:normalized-w-equation}
\frac{1}{v_j}\MA_{\theta_t+\varepsilon_j\omega_t}(w_{j,t})
=\frac{e^{w_{j,t}}F_{B,t}}{J_{j,t}}\,d\nu_t
=:d\mu_{j,t}.
\end{equation}
By the definition of $J_{j,t}$, $\mu_{j,t}$ is a probability measure,
and by \eqref{common:J-bounds} its density is bounded in $L^q(\nu_t)$
uniformly in $j$ and $t$.

We first show that $w_{j,t}$ has minimal singularities. By
\cite[Theorems~3.1 and~4.1]{BEGZ}, there is, for the fixed measure
$\mu_{j,t}$, a potential $\widetilde w_{j,t}$ with minimal
singularities such that
$
 \MA_{\theta_t+\varepsilon_j\omega_t}(\widetilde w_{j,t})
 =v_j\,\mu_{j,t}
 =\MA_{\theta_t+\varepsilon_j\omega_t}(w_{j,t}).
$
Both potentials have full mass. Hence, by the uniqueness statement
in \cite[Theorem~3.1]{BEGZ}, they differ by a constant, and
$w_{j,t}$ also has minimal singularities.

We now apply the envelope estimate \cite[Theorem~1.9]{DGG} to
$w_{j,t}$, which satisfies $\sup_{X_t}w_{j,t}=0$, and obtain
\[
 w_{j,t}\ge V_{\theta_t+\varepsilon_j\omega_t}-C,
\]
where $C$ depends only on $n$, $q$, the constants in
\eqref{common:skoda} and the bound for the density in
\eqref{common:J-bounds}. In particular, $C$ is independent of $j$ and
$t$. This estimate concerns normalized Monge--Amp\`ere measures and
requires no positive lower bound for $v_j$. In a
K\"ahler class a potential has minimal singularities if and only if
it is bounded; this is the case needed in Section~\ref{sec:nef-proof}.

Since $u_{j,t}=w_{j,t}-\log J_{j,t}$, the inequality $w_{j,t}\le0$
and the lower bound for $J_{j,t}$ give an upper bound for $u_j$, and
the envelope estimate and the upper bound for $J_{j,t}$ give a lower
bound. Namely,
\begin{equation}\label{common:uniform-envelopes}
 u_j\le C_1,\qquad
 u_{j,t}\ge V_{\theta_t+\varepsilon_j\omega_t}-C_2,
\end{equation}
where we may take $C_1=\alpha^{-1}\log H$ and $C_2=C+\log C_F$.

Next we derive a lower bound in terms of the given section. Put
\[
 \chi_s=\frac1m\log|s|_{h_A^m}^{2},\qquad
 a_s=\sup_{X_0}\chi_s.
\]
By the Poincar\'e--Lelong formula we have
$\theta_0+\ddc\chi_s=m^{-1}[\operatorname{div}(s)]\ge0$.
Hence $\chi_s-a_s$ is a competitor in the definition of
$V_{\theta_0+\varepsilon_j\omega_0}$, and by
\eqref{common:uniform-envelopes} we obtain
\begin{equation}\label{common:approximate-barrier}
 u_{j,0}\ge\chi_s-a_s-C_2 .
\end{equation}

\medskip\noindent
\emph{Step 3: The limiting metric.}
Consider the upper envelopes on the total space
\[
 U_N=(\sup_{j\ge N}u_j)^*,
 \qquad \psi=\lim_{N\to\infty}U_N.
\]
For $j\ge N$ we have $\varepsilon_j\le\varepsilon_N$; hence $u_j$ is
also $(\theta+\varepsilon_N\omega)$-psh, and, by the common upper
bound in \eqref{common:uniform-envelopes}, $U_N$ is
$(\theta+\varepsilon_N\omega)$-psh. The sequence $(U_N)$ is
decreasing, since the supremum is taken over smaller and smaller sets
of indices. On a coordinate chart we choose smooth functions $a,r$
with $\ddc a=\theta$ and $\ddc r=\omega$ and with $r\ge0$. Then the
functions $U_N+a+\varepsilon_Nr$ are psh and decreasing. Since the
manifold $\X$ is connected, the limit $\psi$ is either $\theta$-psh or
identically $-\infty$.

The lower bound given by the section excludes the second possibility.
In fact, every $u_{j,0}$ satisfies \eqref{common:approximate-barrier},
and this lower bound is preserved when we take a supremum and then
the upper regularization. Hence it holds for $U_N|_{X_0}$, and
therefore for $\psi_0$. In particular, $\psi$ is finite at every
point of $X_0$ where $s$ does not vanish. Thus
$\psi\in\PSH(\X,\theta)$, and \eqref{common:section-barrier} holds
with $C_s=a_s+C_2$.

Finally suppose that $[\theta_t]$ is pseudoeffective for every $t$.
In this case we need no section. Every competitor for $V_{\theta_t}$
is also a competitor for $V_{\theta_t+\varepsilon_j\omega_t}$; hence
$V_{\theta_t+\varepsilon_j\omega_t}\ge V_{\theta_t}$, and the above
argument gives $\psi_t\ge V_{\theta_t}-C_2$ on every fiber. On the
other hand, $\psi_t-C_1$ is $\theta_t$-psh and nonpositive, so that
$\psi_t\le V_{\theta_t}+C_1$. Thus $\psi_t$ has minimal
singularities.
\end{proof}

\subsection{From the limiting metric to extension}

The lower bound \eqref{common:section-barrier} gives the
integrability needed for the extension. In fact, if we multiply the
squared norm of $s$ by $e^{-(m-1)\psi}$, we obtain a function bounded
by a constant multiple of $|s|^{2/m}$, and the remaining boundary
weight $F_B$ is integrable, since every coefficient of $B$ is less
than one. The following proposition is a consequence of the extension
theorem of Zhou--Zhu for singular metrics
\cite[Theorem~1.1 and Remark~1.3]{ZZ}. In the proof we specify the
adjoint bundle and verify the integrability condition of that
theorem.

\begin{proposition}
\label{common:weighted-extension}
Let $m$ be admissible, and suppose that for every nonzero
$s\in H^0(X_0,mA_0)$ there is a metric $h_Ae^{-\psi}$ of positive
curvature, that is, with $\psi\in\PSH(\X,\theta)$, satisfying
\eqref{common:section-barrier}. Then every section of $mA_0$ extends
over a neighborhood of $X_0$. If this hypothesis is satisfied at
every parameter, $X_0$ being replaced by the corresponding fiber,
then $p_*\mathcal O_\X(mA)$ is locally free and commutes with base
change.
\end{proposition}

\begin{proof}
\emph{Step 1: The adjoint metric and the $L^2$ condition.}
First assume that $m\ge2$, and consider
\[
 L_m=mA-K_{\X/\Delta}=(m-1)A+B+M.
\]
Although $(m-1)A$ need not be integral, $L_m$ is a holomorphic line
bundle, since $m$ is admissible. Besides the smooth metric
$h_A^mh_K^{-1}$, it carries the singular metric
\[
 H_m=h_A^mh_K^{-1}e^{-(m-1)\psi-\phi_B},
\]
whose curvature is
\begin{equation}\label{common:auxiliary-curvature}
 \Theta(H_m)=(m-1)(\theta+\ddc\psi)+[B]+\beta\ge0.
\end{equation}
By \eqref{common:section-barrier} the weight of $H_m$ is not
identically $-\infty$ on $X_0$, so that $H_m$ restricts to $X_0$. We
regard $s$ as a section of $K_{X_0}+L_m|_{X_0}$. Up to the fixed
normalization of the canonical metric, its $L^2$ norm satisfies
\begin{align}
 \int_{X_0}|s|_{h_A^m}^{2}
      e^{-(m-1)\psi_0-\phi_{B,0}}\omega_0^n
 &\le C_s'
      \int_{X_0}|s|_{h_A^m}^{2/m}F_{B,0}\omega_0^n
       <\infty .                                      \label{common:L2-bound}
\end{align}
The last integral is finite by \eqref{common:boundary-Lq}, since
$|s|_{h_A^m}$ is bounded on the compact fiber.

\medskip\noindent
\emph{Step 2: Extension.}
Choose a smaller disk $U=\{|t|<r\}\Subset\Delta$, and put
$\X_U=p^{-1}(U)$. Then $\X_U$ is K\"ahler and weakly pseudoconvex. In
fact, $\rho(t)=-\log(1-|t|^2/r^2)$ is a smooth subharmonic exhaustion
function of $U$, and $\rho\circ p$ is a psh exhaustion function of
$\X_U$, since $p$ is proper. Since $p$ is a submersion, $X_0=\{t=0\}$
is a smooth hypersurface along which $dt$ does not vanish; in
particular, its normal bundle is trivial.

We apply the extension theorem of Zhou--Zhu for singular metrics
\cite[Theorem~1.1 and Remark~1.3]{ZZ}, in the formulation for fibers
\cite[Theorem~3.1 and Remark~3.2]{LPWZ}, to the line bundle $L_m$
with the metric $H_m$. Its curvature is semipositive by
\eqref{common:auxiliary-curvature}, and, by adjunction, the $L^2$
integral on $X_0$ required in that theorem is the one in
\eqref{common:L2-bound}, up to the fixed normalization of the
canonical metric. Hence $s\wedge dt$ extends to a holomorphic section
$\widetilde S$ of $K_{\X_U}\otimes L_m$ on $\X_U$. By the definition
of $L_m$ we have
$
 K_{\X_U}\otimes L_m
 \simeq \mathcal O_{\X_U}(mA)\otimes p^*K_U,
$
and $dt$ is a nowhere vanishing frame of $K_U$. Hence
$\widetilde S=\widetilde s\otimes p^*dt$ for a unique section
$\widetilde s$ of $mA$ on $\X_U$, and $\widetilde s|_{X_0}=s$.

If $m=1$, then every $b_i$ is an integer by admissibility, so that
$B=0$, since $0\le b_i<1$; moreover, $M$ is integral. Then $L_1=M$,
and we apply the same extension theorem directly to the smooth
semipositive metric of $M$, without the limiting weight $\psi$.

\medskip\noindent
\emph{Step 3: Local freeness and base change.}
Suppose first that $H^0(X_0,mA_0)\ne0$, and extend a basis of this
space. We choose finitely many points of $X_0$ and local frames such
that the evaluation matrix of this basis has a nonzero maximal minor.
Since $p$ is a submersion, these points extend to local holomorphic
sections of $p$, and the same minor does not vanish for nearby
parameters. Hence the extended sections are linearly independent,
and $h^0(X_t,mA_t)\ge h^0(X_0,mA_0)$ near $0$. The reverse inequality
follows from upper semicontinuity; if the central space is zero,
upper semicontinuity alone gives local constancy. Applying the same
argument at every parameter, we see that $h^0(X_t,mA_t)$ is locally
constant. By Grauert's theorem \cite{Grauert}, the direct image is
locally free, and each of its fibers is identified with the
corresponding space of sections. Since $p$ is smooth and
$\mathcal O_\X(mA)$ is a line bundle, $\mathcal O_\X(mA)$ is flat over
the disk. Hence, by the base change theorem for analytic cohomology
\cite[Chapter~III, Theorem~3.4(b),(d)]{BS}, the direct image commutes
with arbitrary holomorphic base change.
\end{proof}

\section{An analytic proof in the nef case}
\label{sec:nef-proof}

In this section we assume that the log-adjoint class is nef on every
fiber. In this case the extension theorem can be proved by means of
conic K\"ahler--Einstein metrics, pluripotential estimates and the
$L^2$ extension theorem alone, without the birational constructions
of the general case. The untwisted case with empty boundary is due to
Deng--Ning--Wang--Zhou \cite[Theorems~4.2 and~5.1]{DNWZ} and
Chen--Rao--Wang \cite[Theorem~1.2]{CRW}. Our normalization argument
follows their method, and for the logarithmic case we use Guenancia's
positivity theorem for families of conic K\"ahler--Einstein metrics
\cite{Gue}.

\begin{theorem}\label{thm:nef}
Let $p:\X\to\Delta$ be a proper holomorphic submersion with
connected compact fibers and K\"ahler total space.  Let
$B=\sum_i b_iD_i$ be a relative SNC divisor with rational
coefficients $0\le b_i<1$, and let $M$ be a rational line bundle
with a smooth Hermitian metric of semipositive curvature on $\X$.
Suppose that
$
 A_t=K_{X_t}+B_t+M_t
$
is nef for every $t\in\Delta$. Then, for every integer $m\ge1$ such
that $mB$ and $mM$ are integral, every section of $mA_0$ extends to a
section of $m(K_{\X/\Delta}+B+M)$ over a neighborhood of $X_0$.
Consequently, $h^0(X_t,mA_t)$ is locally constant, and the direct
image $p_*\mathcal O_\X(m(K_{\X/\Delta}+B+M))$ is locally free and
commutes with holomorphic base change.
\end{theorem}

We assume only that the adjoint bundle $A_t$ is nef on every fiber.
No positivity of $K_{X_t}$ or of $B_t$ separately is assumed, and
$\vol(c_1(A_t))$ may be zero. In the proof we verify the hypotheses
of Lemma~\ref{common:metric-criterion}. The perturbed classes are
K\"ahler, and hence their volumes are constant; the positivity of the
perturbed conic metrics on the total space follows from Guenancia's
theorem (Proposition~\ref{nef:conic-input}).

\subsection{Positive variation of conic metrics}

If $A_t$ is nef, then every class $[\theta_t+\varepsilon\omega_t]$
with $\varepsilon>0$ is K\"ahler. We keep the reference data
\eqref{common:reference-data}; in particular,
$F_B=\prod_i|\sigma_i|_{h_i}^{-2b_i}\ge1$ and $\beta=\Theta(h_M)\ge0$.
We apply \cite[Section~1.3 and Section~4]{Gue} with the twisting form
$\gamma=\beta+\varepsilon\omega$ and with cone angles $2\pi(1-b_i)$
along $D_i$, the components with $b_i=0$ being omitted. Then the
background form of Guenancia's equation on the fibers is
$\Theta(h_K)+\sum_i b_i\Theta(h_i)+\gamma=\theta+\varepsilon\omega$.
The construction in \cite[Section~1.3]{Gue} and the positivity
theorem in \cite[Section~4]{Gue} are formulated for a proper smooth
K\"ahler morphism onto a disk; hence we need no compactification.

\begin{proposition}[{\cite[Section~1.3 and Section~4]{Gue}}]
\label{nef:conic-input}
For every $\varepsilon>0$ there is a locally bounded function
$\phi_\varepsilon$ on $\X$ such that
\[
 \theta+\varepsilon\omega+\ddc\phi_\varepsilon\ge0
\]
and, on every fiber,
\begin{equation}\label{nef:cone-equation}
 (\theta_t+\varepsilon\omega_t+
      \ddc\phi_{\varepsilon,t})^n
 =e^{\phi_{\varepsilon,t}}F_{B,t}\omega_t^n.
\end{equation}
Its restriction to each fiber is the unique bounded solution
of this equation.
\end{proposition}

For each fixed $\varepsilon$ the local constructions agree on
overlaps, since the bounded solutions on the fibers are unique. Hence
all the $\phi_\varepsilon$ are defined over the same disk, although
their local bounds may depend on $\varepsilon$.

\subsection{Proof of the nef extension theorem}

\begin{proof}[Proof of Theorem~\ref{thm:nef}]
We may assume that $n\ge1$. Fix an admissible integer $m$.

\medskip\noindent
\emph{Step 1: The hypotheses of Lemma~\ref{common:metric-criterion}.}
Let $0<\varepsilon\le1$. Since $[\theta_t+\varepsilon\omega_t]$ is
K\"ahler, its volume
\[
 v_\varepsilon=\int_{X_t}(\theta_t+\varepsilon\omega_t)^n
\]
is a top intersection number, although the form
$\theta_t+\varepsilon\omega_t$ itself need not be positive. Hence
$v_\varepsilon>0$. Moreover, $v_\varepsilon$ is independent of $t$,
since $\theta+\varepsilon\omega$ is closed on $\X$ and fiber
integration commutes with exterior differentiation. By
Proposition~\ref{nef:conic-input} there is a function
$\phi_\varepsilon\in\PSH(\X,\theta+\varepsilon\omega)$ whose
restrictions to the fibers are bounded solutions of
\eqref{nef:cone-equation}, and bounded potentials in a K\"ahler class
have full mass. Therefore both hypotheses of
Lemma~\ref{common:metric-criterion} are satisfied for every sequence
$\varepsilon_j\downarrow0$. In Step~2 of the proof of that lemma the
normalized potentials are bounded. Hence we use
\cite[Theorem~1.9]{DGG} only for K\"ahler classes, and we need no
positive lower bound for $v_\varepsilon$.

\medskip\noindent
\emph{Step 2: Conclusion.}
By Lemma~\ref{common:metric-criterion}, for every nonzero section
$s\in H^0(X_0,mA_0)$ there is a function $\psi\in\PSH(\X,\theta)$
with
\[
 \psi|_{X_0}\ge\frac1m\log|s|_{h_A^m}^2-C_s.
\]
Applying Proposition~\ref{common:weighted-extension} to the line
bundle
\[
 L_m=mA-K_{\X/\Delta}=(m-1)A+B+M,
\]
we see that $s$ extends; if $m=1$, then $B=0$ by admissibility, and
the smooth semipositive metric on $M$ suffices. The same argument
applies to a disk centered at any parameter. Hence the last assertion
of Proposition~\ref{common:weighted-extension} gives the local
constancy of $h^0(X_t,mA_t)$, the local freeness and the base change
property.
\end{proof}

\begin{remark}\label{nef:scope}
By the last assertion of Lemma~\ref{common:metric-criterion}, the
above argument also gives a positively curved metric on $A$ whose
restriction to every fiber has minimal singularities. We have used
neither the bigness of $A_t$ nor a minimal model. In order to apply
Theorem~\ref{thm:nef} to rounded boundaries as in
Corollary~\ref{cor:round}, we must verify that
$K_{X_t}+\lfloor mB_t\rfloor/m+M_t$ itself is nef, since nefness is
in general not preserved when the boundary is decreased.
\end{remark}

\section{Relative models and invariance of adjoint volumes}
\label{sec:relative-log-geometry}

For the analytic argument we need a model on which a big perturbation
of the log-adjoint class splits into a semipositive class and an
effective exceptional divisor. Such a model is constructed near a
given fiber in Proposition~\ref{prop:local-log-model}. Here the
boundary may have components in common with the divisorial negative
part of the adjoint class, whereas the deformation theorem for the
minimal model program (MMP) requires that these divisors have no
common component. We therefore first remove their common part and
construct the relative model for the reduced boundary. We then show
that the removed divisor is exceptional on this model, so that it can
be restored.

The exceptionality gives the orthogonality needed in
Section~\ref{sec:signed-analysis}. The model also gives the local
constancy of the volume, and an open--closed argument then yields a
fixed disk on which the perturbation can tend to zero. The same
argument proves Theorem~\ref{thm:real-log-volume} in the klt case;
the case of coefficients equal to one then follows by approximation.

All equalities of classes in this section are equalities in the real
Bott--Chern cohomology $H^{1,1}_{\mathrm{BC}}(\cdot,\R)$; on the
normal complex spaces below, Bott--Chern classes and their pullbacks
are understood in the sense of \cite{HX}. By a \emph{modification}
we mean a proper bimeromorphic morphism, and by a bimeromorphic
contraction we mean a bimeromorphic map whose inverse does not
contract any divisor. A prime divisor is said to be exceptional for a
modification if its image has codimension at least two, and a
contraction is said to be small if it has no exceptional prime
divisors.

We recall the singularities of pairs \cite[Section~2.3]{KM98} in the
form needed here. Let $X$ be a complex manifold, and let $B$ be an
effective real divisor on $X$ with simple normal crossing support.
For a modification $g:X'\to X$ from a complex manifold we define
$B_{X'}$ by
\[
 K_{X'}+B_{X'}=g^*(K_X+B),\qquad g_*B_{X'}=B.
\]
For a prime divisor $F$ on $X'$ we call
$a(F;X,B)=-\operatorname{coeff}_F(B_{X'})$ the \emph{discrepancy} of
$(X,B)$ along $F$, and $A_{X,B}(F)=1+a(F;X,B)$ its \emph{log
discrepancy}. We say that $(X,B)$ is \emph{klt} if every coefficient
of $B$ is less than one and $a(F;X,B)>-1$ for every exceptional prime
divisor $F$ of every such modification, \emph{log canonical} if every
coefficient is at most one and $a(F;X,B)\ge-1$ for all such $F$, and
\emph{canonical} if $a(F;X,B)\ge0$ for all such $F$. For pairs with
simple normal crossings, $(X,B)$ is klt if and only if every
coefficient of $B$ is less than one \cite[Corollary~2.31(3)]{KM98}. A
modification $\tau:\widetilde X\to X$ is said to be \emph{crepant}
for pairs $(\widetilde X,\widetilde B)$ and $(X,B)$ if
$K_{\widetilde X}+\widetilde B=\tau^*(K_X+B)$. Generalized pairs
carry, in addition, a nef part. For their discrepancies, for
generalized klt and canonical pairs, and for the terminology of the
K\"ahler minimal model program, we refer to \cite{HX}.

Let $\xi$ be a pseudoeffective class on a compact K\"ahler manifold.
The divisorial negative part of $\xi$ is defined by
\[
 N(\xi)=\sum_D\nu(\xi,D)D,
\]
where $D$ runs over the prime divisors and $\nu(\xi,D)$ is the
minimal multiplicity of $\xi$ along $D$ \cite{BouZ}. Namely, if
$\kappa$ is a K\"ahler form and $T_{\min,\epsilon}$ is a current with
minimal singularities in $\xi+\epsilon[\kappa]$, then
$\nu(\xi,D)=\lim_{\epsilon\downarrow0}\nu(T_{\min,\epsilon},D)$,
where $\nu(T,D)$ denotes the generic Lelong number of $T$ along $D$,
that is, the coefficient of $[D]$ in the Siu decomposition of $T$.
The decomposition $\xi=\bigl(\xi-[N(\xi)]\bigr)+[N(\xi)]$ is called
the divisorial Zariski decomposition of $\xi$ \cite{BouZ}. For
effective divisors $G=\sum_Dg_DD$ and $H=\sum_Dh_DD$ we write
\[
 G\wedge H:=\sum_D\min(g_D,h_D)D
\]
for the coefficientwise minimum (this is not a product of currents).
Thus $G\wedge H=0$ means that $G$ and $H$ have no common prime
component.

\subsection{Extension of the central minimal model program}\label{sec:inputs}

We shall use the extension theorem of Hacon--Xie
\cite[Theorem~7.7]{HX}, whose proof uses the contraction construction
in \cite[proof of Proposition~4.16]{HLR}. These results extend the
central MMP steps together with their contraction diagrams. We state
them in the following form, which also contains the flatness and
normality properties needed below. In \cite[Theorem~7.7]{HX} the
projectivity hypothesis in the original statement of
\cite[Proposition~4.16]{HLR} is replaced by a K\"ahler hypothesis.

\begin{lemma}
\label{lem:relative-mmp-interface}
Let $p:V\to U$ be a proper holomorphic submersion with connected
compact fibers and K\"ahler total space, where $U$ is a relatively
compact disk. Let $\Gamma$ be an effective real relative SNC
boundary, and let $\beta$ be a smooth closed K\"ahler form on $V$,
regarded as the descending nef part of a generalized pair. Put
\[
 a_V=c_1(K_{V/U}+\Gamma)+[\beta].
\]
Assume that $0\in U$, that
$a_{V,0}$ is big, that the central generalized pair is canonical,
and that
\[
 \lfloor\Gamma_0\rfloor=0,
 \qquad N(a_{V,0})\wedge\Gamma_0=0.
\]
Fix a finite sequence of central transcendental MMP steps for this
adjoint class, and write its models as
$V_0^{(0)}=V_0,\ldots,V_0^{(k)}$, where $k$ is the number of steps.
Then, after shrinking $U$, the sequence extends to families
$V^{(0)}=V,\ldots,V^{(k)}$ with the following properties:
\begin{enumerate}
\item The intermediate spaces $V^{(j)}$ are normal proper flat
families over $U$, and their central fibers are the given models
$V_0^{(j)}$. Let $\Gamma^{(j)}$ be the transformed boundary and
$\mathfrak b^{(j)}$ the trace of the nef part on $V^{(j)}$.
They define the real Bott--Chern adjoint class
\[
 a_{V^{(j)}}:=c_1(K_{V^{(j)}/U}+\Gamma^{(j)})+\mathfrak b^{(j)},
\]
whose restriction to $V_0^{(j)}$ is the corresponding central
adjoint class.
\item Each central divisorial step extends to a proper bimeromorphic
contraction morphism with the given central restriction. Each
flipping step extends to a diagram of contraction morphisms
$V^{(j)}\to T^{(j)}\leftarrow V^{(j+1)}$ with common target
$T^{(j)}$. The central restrictions of the two arrows are the given
contraction and its flip; in particular, the two
central flipping contractions are small.
\item Let $V':=V^{(k)}$ be the final model over $U$, with adjoint
class $a_{V'}:=a_{V^{(k)}}$.
Then the composite $V\dashrightarrow V'$ is an
$a_V$-negative bimeromorphic contraction. On a common resolution
$V\xleftarrow{\mu} Z\xrightarrow{r}V'$ we have, by negativity,
\[
 \mu^*a_V=r^*a_{V'}+[F],
 \qquad F\ge0,
 \qquad F\text{ is }r\text{-exceptional}.
\]
\end{enumerate}
Moreover, after further shrinking, the final family and the finitely
many intermediate families have normal fibers.
\end{lemma}

\begin{proof}
The hypotheses of \cite[Theorem~7.7]{HX} are satisfied. Since the
boundary does not contain the central fiber and the nef part
descends, the generalized discrepancies on that fiber coincide with
the ordinary discrepancies. The contraction diagrams are obtained
from the cited theorem together with
\cite[proof of Proposition~4.16]{HLR}, which is used in its proof.
In that construction the central contraction is first extended to a
morphism by means of \cite[Theorem~4.7]{HLR}. In the flipping case,
local flips over Stein neighborhoods of the target glue to a
morphism, and the comparison of discrepancies identifies its central
restriction with the prescribed flip. In the divisorial case, the
relative log canonical model is isomorphic to the target of the
contraction near the central fiber, and we obtain a contraction
morphism. Repeating this construction, we obtain the diagrams in (2)
and the negative relative sequence.

By relative adjunction, the restrictions of the adjoint classes are
the prescribed central classes. Each normal model dominates the
smooth base curve and is therefore flat over it. The central fibers
are normal by the construction of the central MMP. Since geometric
normality is an open condition in a proper flat family, all fibers
are normal after one further shrinking, common to the finitely many
families. Finally, we pull back the effective exceptional errors of
all the negative steps to a common resolution and add them up. This
gives the equality in (3).
\end{proof}

In order to extend the contraction of the final central model, we
shall also use \cite[Lemma~4.18]{LPWZ}, which we state in the
following form. Let $h_0:X'_0\to Y_0$ be a bimeromorphic morphism
between normal compact K\"ahler spaces with rational singularities,
let $X'\to U$ be a proper flat normal deformation of $X'_0$, and let
$\gamma_0$ be a K\"ahler class on $Y_0$. If
$a_{X'}\in H^{1,1}_{\mathrm{BC}}(X',\R)$ satisfies
$a_{X'}|_{X'_0}=h_0^*\gamma_0$, then, after shrinking, there exist a
proper flat normal family $Y\to U$ with normal fibers, a proper
bimeromorphic morphism $h:X'\to Y$ extending $h_0$, and a K\"ahler
form $\gamma$ on $Y$ such that
\[
 a_{X'}=h^*[\gamma],\qquad [\gamma]|_{Y_0}=\gamma_0.
\]
This is a semiampleness statement for the Bott--Chern class
$a_{X'}$.

\subsection{A simultaneous canonicalization}

Lemma~\ref{lem:relative-mmp-interface} requires that the central pair
be canonical. In the following lemma we make all the fiber pairs
canonical simultaneously, for a relative SNC klt boundary. We use the
discrepancy formula for blowups and the canonicity criterion for SNC
pairs \cite[Lemma~2.29 and Corollary~2.31(3)]{KM98}. The construction
preserves the spaces of sections in every admissible degree.

\begin{lemma}\label{lem:canonicalization}
Let \(p:\X\to\Delta\) be a proper holomorphic submersion with K\"ahler
total space, and let \(B=\sum b_iD_i\) be an effective rational
relative SNC divisor with \(0\leq b_i<1\). After shrinking the disk,
there is a finite composition of blowups
\[
 \tau:\widetilde\X\longrightarrow\X
\]
along smooth relative boundary strata such that
\[
 K_{\widetilde\X/\Delta}+\widetilde B
     =\tau^*(K_{\X/\Delta}+B).
\]
Here $\widetilde B$ denotes the boundary determined by this crepant
equality; it is effective and relative SNC, with coefficients less
than one. Moreover, $p\circ\tau$ is smooth, $\widetilde\X$ is
K\"ahler, and every fiber pair $(\widetilde X_t,\widetilde B_t)$ is
canonical.
For every integer $m\ge1$ with $mB$ integral, $m\widetilde B$ is
also integral, and the pullback identifies the sections of
$m(K_{\X/\Delta}+B)$ and
$m(K_{\widetilde\X/\Delta}+\widetilde B)$, both on the total
spaces and on each fiber.
The geometric assertions hold also for real coefficients
$0\le b_i<1$; only the last assertion, concerning integral degrees,
requires rational coefficients.
\end{lemma}

\begin{proof}
\emph{Step 1: Crepant blowups.}
As long as two boundary components \(D_i\) and \(D_j\) with
\(b_i+b_j>1\) meet, we blow up their intersection, one connected
component at a time if necessary. The center is smooth over the disk
and has codimension two in every fiber. We denote by $E$ the
exceptional divisor of this blowup and give it the coefficient
\[
 b_E=b_i+b_j-1\in(0,1).
\]
By the discrepancy formula for blowups
\cite[Lemma~2.29, pp.~52--53]{KM98}, this
choice is crepant. The new boundary is again relative SNC, and the
new total space is again smooth over the disk.

\medskip\noindent
\emph{Step 2: Termination.}
The new exceptional divisors may again meet components with large
coefficients, and we must show that the process terminates. In a
boundary chart of the original family, let $D_1,\ldots,D_r$ be the
boundary components meeting the chart, and put $a_i:=1-b_i>0$. The
toroidal divisor corresponding to a primitive vector
\(v=(v_1,\ldots,v_r)\in\mathbf Z_{\geq0}^r\) has log discrepancy
\[
 A(v)=\sum_i a_iv_i.
\]
Each blowup is a star subdivision which introduces a ray with
\(A(v)<1\). Since \(v_i<1/a_i\) for every \(i\), there are only
finitely many such rays in this chart. After shrinking around the
reference fiber, properness and the relative SNC condition give a
finite boundary cone complex. Hence only finitely many rays can occur
in the whole construction. Each blowup introduces a ray which is not
yet in the current regular fan, and subdivisions never remove rays.
Therefore the procedure terminates.

\medskip\noindent
\emph{Step 3: Canonicity.}
On the final model, any two boundary components that meet have
coefficient sum at most one, and every coefficient is less than one.
Hence each fiber pair is canonical by the discrepancy criterion for
SNC pairs \cite[Corollary~2.31(3), p.~53]{KM98}. In fact, this
criterion bounds the exceptional discrepancies from below by the
minimum of $1$, the numbers $1-b_i$, and the numbers $1-b_i-b_j$
for distinct components which meet. All these numbers are
nonnegative. The criterion is local in SNC coordinates, and therefore
it applies also in the analytic setting. For real coefficients we
approximate each coefficient from below by nonnegative rational
numbers. The pairwise inequalities are preserved, and the discrepancy
of each fixed exceptional divisor depends continuously on the
boundary coefficients. Passing to the limit, we obtain canonicity
also in the real case.

\medskip\noindent
\emph{Step 4: K\"ahlerness and sections.}
Each blowup is a projective morphism with smooth center. To obtain a
K\"ahler form on the blown-up space, we add to the pulled-back
K\"ahler form a small multiple of the curvature of a metric on
$\mathcal O(-E)$ which is positive along the fibers of the
exceptional divisor $E$. By compactness over a smaller closed disk,
this choice can be made uniformly, so that the total space remains
K\"ahler after shrinking. The discrepancy and termination arguments
use only $1-b_i>0$, and therefore the construction works also for
real coefficients.

The recursion \(b_E=b_i+b_j-1\) introduces no new denominators.
Thus \(m\widetilde B\) is integral if \(mB\) is, and the crepant
equality gives an isomorphism of the line bundles of degree \(m\).
Since \(\tau_*\mathcal O_{\widetilde\X}=\mathcal O_\X\), and similarly
for every \(\tau_t\), the projection formula identifies the spaces of
sections. In particular, a section extended on $\widetilde\X$
descends to an extension of the original section.
\end{proof}

Let $\beta$ be a smooth closed semipositive twisting form, and let
$\omega$ be a K\"ahler form on $\X$. For the proof of the extension
theorem we may replace the family by its canonicalization, $\beta$
by its pullback, and $\omega$ by a K\"ahler form on the new total
space; we keep the notation $p:\X\to\Delta$ and $B$. In what follows
every fiber pair $(X_t,B_t)$ is canonical, and $\beta$ remains fixed
while the perturbation parameter tends to zero.

\subsection{Removing the central fixed boundary}

Our boundary need not satisfy the condition
$N(a_{V,0})\wedge\Gamma_0=0$ of
Lemma~\ref{lem:relative-mmp-interface}. We therefore remove the
common part of the boundary and the divisorial negative part; by the
definition of the negative part, this common part is contained in
every positive current in the central class. A similar reduction by
fixed divisors is used in the projective proof of
\cite[Theorem~4.2]{HMX18}; here we use the divisorial Zariski
decomposition of real Bott--Chern classes.

Fix $\varepsilon>0$, consider the perturbed adjoint class
\[
 \alpha:=c_1(K_{\X/\Delta}+B)+[\beta]+\varepsilon[\omega],
\]
and assume that $\alpha_0$ is big. For brevity we suppress the
dependence of $\alpha$ and of the models below on $\varepsilon$.

The following relative reduction uses the description of minimal
multiplicities in \cite[Proposition~3.6]{BouZ} and the volume
identity in \cite[Proposition~3.20]{BouZ}. We give the argument for
the partial negative divisor and for its extension along the family.

\begin{lemma}\label{lem:fixed-boundary}
Put, coefficientwise on the central fiber,
\[
 C_0=B_0\wedge N(\alpha_0).
\]
Thus the coefficient of $D_{i,0}$ in $C_0$ is
$c_i:=\min\{b_i,\nu(\alpha_0,D_{i,0})\}$. We extend $C_0$ to the
divisor $C:=\sum_i c_iD_i$ on $\X$, and put
\[
 B'=B-C,\qquad
 \alpha':=\alpha-[C]
          =c_1(K_{\X/\Delta}+B')+[\beta]+\varepsilon[\omega].
\]
Then \(B'\) is an effective real relative SNC boundary,
every \((X_t,B'_t)\) is canonical, and
\[
 N(\alpha'_0)=N(\alpha_0)-C_0,\qquad
 N(\alpha'_0)\wedge B'_0=0,\qquad
 \vol(\alpha'_0)=\vol(\alpha_0)>0.
\]
Moreover, the minimal singularity types in \(\alpha_0\) and
\(\alpha'_0\) differ by the divisor \(C_0\).
\end{lemma}

\begin{proof}
By \cite[Proposition~3.6]{BouZ}, the generic Lelong numbers of every
positive current \(T\in\alpha_0\) are not less than the corresponding
coefficients of \(N(\alpha_0)\). Hence the Siu decomposition of $T$
gives \(T-[C_0]\geq0\). Consequently, subtracting and adding
\([C_0]\) give mutually inverse bijections between the positive
currents in \(\alpha_0\) and those in \(\alpha'_0\). The minimal
singularity types in these classes differ by \([C_0]\). Since a
divisorial current has zero absolutely continuous part, corresponding
currents satisfy $(T-[C_0])_{\mathrm{ac}}=T_{\mathrm{ac}}$. Taking
the supremum of the integrals of their top exterior powers in the
definition of volume gives $\vol(\alpha'_0)=\vol(\alpha_0)$. This
is the version for a partial divisor of the argument in
\cite[proof of Proposition~3.20]{BouZ}. In particular, $\alpha'_0$ is
big. Applying \cite[Proposition~3.6]{BouZ} to currents with minimal
singularities in both big classes, we obtain
$N(\alpha'_0)=N(\alpha_0)-C_0$.

Along each boundary component we have removed the smaller of its
boundary coefficient and its coefficient in \(N(\alpha_0)\). Hence
at least one of the two remaining coefficients is zero, and this
proves \(N(\alpha'_0)\wedge B'_0=0\). Since canonicity is preserved
when an effective boundary is lowered, every \((X_t,B'_t)\) is
canonical.

Finally, the extension of $C_0$ is well defined. In fact, each prime
boundary component $D_i\to\Delta$ is proper and smooth, and, by Stein
factorization over the simply connected disk, its fibers are
connected and hence irreducible. Thus each $D_{i,0}$ is a prime
divisor, $C|_{X_0}=C_0$, and $0\le C\le B$, since $0\le c_i\le b_i$.
Since the coefficients $c_i$ may be real, we use for the intermediate
boundary $B'$ the results on generalized pairs with real
coefficients. The original rational boundary $B$ is restored before
sections are extended.
\end{proof}

\subsection{The central model and exceptionality of the removed divisor}

The reduced boundary satisfies the hypotheses of
Lemma~\ref{lem:relative-mmp-interface}. Before extending the MMP
over the base, we construct its central model and prove that the
resulting map contracts every component of $C_0$.

We run the MMP with scaling for the generalized pair
$(X_0,B'_0+\beta_0+\varepsilon\omega_0)$. The required hypotheses are
satisfied: \(\alpha'_0\) is big, the pair is generalized klt, and the
nef part $\beta_0+\varepsilon\omega_0$ is a smooth K\"ahler form.
In particular, the boundary plus the nef part is big. Moreover,
$X_0$ is smooth and hence strongly $\Q$-factorial, as required in the
formulation with scaling. A sufficiently large multiple of a K\"ahler
class gives a scaling class whose addition makes the initial adjoint
class nef. Therefore, by the existence and termination results
\cite[Theorems~1.1 and~1.5]{HX}, we obtain a finite MMP
\[
 X_0\dashrightarrow X'_0
\]
to a good model $X'_0$. We denote by $B'_{X'_0}$ the transformed
boundary and by $\mathfrak b_{X'_0}$ the trace of the nef part, and
put
\[
 a_{X'_0}:=c_1(K_{X'_0}+B'_{X'_0})+\mathfrak b_{X'_0}.
\]
The class $a_{X'_0}$ is nef. It is also big. In fact, on a common
resolution the pullback of $\alpha'_0$ is the sum of the pullback of
$a_{X'_0}$ and an effective exceptional divisor; hence $a_{X'_0}$ is
the pushforward of a big class. The transformed boundary plus the nef
part is modified big in the sense of \cite[Definition~2.8]{HX}. To
see this, choose a common smooth K\"ahler resolution
$X_0\xleftarrow{\rho_0}\widehat W_0\xrightarrow{\sigma_0}X'_0$,
and let $\widehat B_0\ge0$ be the strict transform of $B'_{X'_0}$.
Then
\[
 [B'_{X'_0}]+\mathfrak b_{X'_0}
 =(\sigma_0)_*\bigl(\rho_0^*([\beta_0]+\varepsilon[\omega_0])
                         +[\widehat B_0]\bigr).
\]
The class inside the pushforward is big, since it is the pullback
of a K\"ahler class plus an effective divisor.

By the transcendental base-point-free theorem
\cite[Theorem~1.4]{HX}, there exist a contraction
\(h_0:X'_0\to Y_0\) onto a normal compact K\"ahler space and a
K\"ahler class \(\gamma_0\) on $Y_0$ such that
$a_{X'_0}=h_0^*\gamma_0$. Since $a_{X'_0}$ is big, we have
$\dim Y_0=n$; since, moreover, the fibers of $h_0$ are connected,
$h_0$ is bimeromorphic. On a common smooth K\"ahler resolution
\[
 X_0\xleftarrow{\mu_0}W_0\xrightarrow{f_0}Y_0
\]
we define the class $\mathfrak p_0:=f_0^*\gamma_0$. The MMP gives
an effective $f_0$-exceptional divisor $E'_0$ such that
\begin{equation}\label{eq:central-reduced-model}
 \mu_0^*\alpha'_0=\mathfrak p_0+[E'_0],
 \qquad \mathfrak p_0=f_0^*\gamma_0,
 \qquad E'_0\geq0,
 \qquad E'_0\text{ is }f_0\text{-exceptional}.
\end{equation}
Here $\mathfrak p_0$ and $\gamma_0$ are classes, and the integrals of
their powers are intersection numbers. The resolution $W_0$ is chosen
on the central fiber, and it need not be the central fiber of the
relative resolution constructed below. The composite
$X_0\dashrightarrow Y_0$ is a bimeromorphic contraction. The spaces
$X'_0$ and $Y_0$ have rational singularities by the results on
generalized klt pairs and crepant contractions in
\cite[Lemmas~2.25 and~2.31]{HX} (see also
\cite[Proposition~4.12(1),(2)]{LPWZ}), which admit a nonzero
boundary.

We now show that this model contracts the removed boundary. For the
proof we combine the perturbation argument in
\cite[proof of Theorem~3.18]{BouZ} with the description of the
non-K\"ahler locus of a nef and big class in
\cite[Theorem~1.1]{CT}. We recall the relevant notions. A closed
$(1,1)$-current is said to have \emph{analytic singularities} if its
local potentials are of the form $c\log\sum_j|g_j|^2+v$, where $c>0$
is a constant, the $g_j$ are finitely many holomorphic functions and
$v$ is smooth. The \emph{non-K\"ahler locus} of a big class on a
compact K\"ahler manifold is the intersection of the sets where $T$ is
not smooth, $T$ ranging over all K\"ahler currents in the class with
analytic singularities \cite{BouZ,CT}. Let $\mathfrak p$ be a nef
class. A positive-dimensional irreducible analytic subset $V$ is said
to be \emph{null} for $\mathfrak p$ if $\int_V\mathfrak p^{\dim V}=0$,
and the \emph{null locus} of $\mathfrak p$ is the union of its null
subsets.

\begin{lemma}\label{lem:central-exceptionality}
Every component of \(\mu_0^*C_0\) is \(f_0\)-exceptional.
\end{lemma}

\begin{proof}
Put \(C_W=\mu_0^*C_0\), and suppose that a prime component \(F\) of
\(C_W\) with coefficient \(c_F>0\) is not \(f_0\)-exceptional. Then
\[
 \int_F \mathfrak p_0^{n-1}
       =\int_{(f_0)_*F}\gamma_0^{n-1}>0,
\]
so that $F$ is not a null divisor for the nef and big class
\(\mathfrak p_0=f_0^*\gamma_0\). By Collins--Tosatti
\cite[Theorem~1.1]{CT}, the non-K\"ahler locus of $\mathfrak p_0$
coincides with its null locus. A prime divisor cannot be contained in
a larger proper irreducible analytic subset of $W_0$, and $W_0$
itself is not null, since $\mathfrak p_0$ is big. Hence $F$ is not
contained in the non-K\"ahler locus. Therefore we can choose a
K\"ahler current \(S\in\mathfrak p_0\) with analytic singularities
which is smooth at a general point of $F$.

Using the strict positivity of $S$, we replace a small part of
$[C_W]$ by a smooth representative. Fix a K\"ahler form
$\omega_{W_0}$ on $W_0$ and a smooth closed representative
$\theta_C$ of $[C_W]$. We choose constants $a_S>0$ and $L_C\ge0$ such
that $S\ge a_S\omega_{W_0}$ and $\theta_C\ge-L_C\omega_{W_0}$. For
$0<\lambda<1$ with $\lambda L_C<a_S$, the current
$S+\lambda\theta_C$ is positive; hence
\[
 T=S+\lambda\theta_C+[E'_0]+(1-\lambda)[C_W]
\]
is a positive current in \(\mu_0^*\alpha_0=\mathfrak p_0+[E'_0]+[C_W]\).

Since \(X_0\dashrightarrow Y_0\) is a bimeromorphic contraction, every
\(\mu_0\)-exceptional divisor is \(f_0\)-exceptional. In particular,
\(F\) is not \(\mu_0\)-exceptional, and hence it is the strict
transform of a component \(C_i\) of \(C_0\). Moreover, $F$ is not a
component of \(E'_0\), and the generic Lelong number of \(S\) along
$F$ is zero. Therefore $(\mu_0)_*T$ is a positive current in
\(\alpha_0\) whose generic Lelong number along \(C_i\) is
\((1-\lambda)c_F<c_F\). This contradicts
\[
 c_F=\operatorname{coeff}_{C_i}(C_0)
      \leq\operatorname{coeff}_{C_i}N(\alpha_0),
\]
since the right-hand side is a lower bound for the generic Lelong
number along $C_i$ of every positive current in \(\alpha_0\).
\end{proof}

\subsection{Extension and restoration of the relative model}

We now extend the central model by
Lemma~\ref{lem:relative-mmp-interface}, and pass from the
exceptionality on the central fiber in
Lemma~\ref{lem:central-exceptionality} to exceptionality on the total
space. Restoring the removed boundary, we shall obtain the desired
decomposition of the original class $\alpha$. For zero boundary this
is the construction of \cite[Theorem~4.19]{LPWZ}. In the proof below
we describe the additional reduction and restoration needed for a
nonzero boundary, together with their consequences for the volume.

\begin{proposition}
\label{prop:local-log-model}
After shrinking the disk about \(0\), there exist a proper flat
family \(g:Y\to\Delta\) with normal total space and normal fibers, a
K\"ahler form \(\gamma\) on \(Y\), and a common resolution
\[
 \X\xleftarrow{\mu} Z\xrightarrow{f}Y
\]
with \(Z\) K\"ahler, such that
\begin{equation}\label{eq:restored-model}
 \mu^*\alpha=[P]+[E_\alpha],
 \qquad P=f^*\gamma,\qquad
 E_\alpha\geq0,\qquad
 E_\alpha\text{ is }f\text{-exceptional}.
\end{equation}
Put $q:=p\circ\mu=g\circ f:Z\to\Delta$. Outside a closed discrete
set of parameters, $q$ is a smooth proper submersion, the induced
maps $\mu_t:Z_t\to X_t$ and $f_t:Z_t\to Y_t$ are modifications, the
relevant divisor supports have relative simple normal crossings, and
each component of $E_{\alpha,t}$ is $f_t$-exceptional. Moreover,
\(\alpha_t\) is big on the whole smaller disk, and its volume is
constant there.
\end{proposition}

\begin{proof}
\emph{Step 1: The reduced model.}
By Lemma~\ref{lem:fixed-boundary}, the central reduced pair is
canonical, and its boundary has zero integral part and no common
component with the divisorial negative part. The nef part
$\beta+\varepsilon\omega$ is a smooth K\"ahler form on the total
space. Hence, by Lemma~\ref{lem:relative-mmp-interface}, the chosen
central MMP extends together with its contraction diagrams. Let
\(X'\to\Delta\) be the resulting normal flat family, and let
\(a_{X'}\) be its real Bott--Chern adjoint class. By relative
adjunction we have
\[
 a_{X'}|_{X'_0}=h_0^*\gamma_0.
\]
Since \(X'_0\) and \(Y_0\) have rational singularities, we may apply
\cite[Lemma~4.18]{LPWZ}, which admits a nonzero boundary. After
shrinking, \(h_0\) extends to a proper bimeromorphic morphism
\(h:X'\to Y\), and
\[
 a_{X'}=h^*[\gamma],\qquad [\gamma]|_{Y_0}=\gamma_0
\]
for a K\"ahler form \(\gamma\) on \(Y\).

We resolve the graph of \(\X\dashrightarrow X'\) by a sequence of
blowups $\mu:Z\to\X$, and denote by $r:Z\to X'$ the induced morphism.
Put $f:=h\circ r$ and $P:=f^*\gamma$. The form $P$ is smooth, closed
and semipositive, and $Z$ is K\"ahler after shrinking the base. By
Lemma~\ref{lem:relative-mmp-interface}(3) we obtain an effective
$r$-exceptional divisor $E'$. It is also $f$-exceptional, since
composition with $h$ cannot increase the dimension of its image.
Hence
\[
 \mu^*\alpha'=[P]+[E'],\qquad E'\geq0,
 \qquad E'\text{ is }f\text{-exceptional}.
\]

\medskip\noindent
\emph{Step 2: Restoring the removed boundary.}
We claim that $\mu^*C$ is also $f$-exceptional. We call a divisor
horizontal if it dominates the base disk, and vertical if it is
contained in a fiber. Let $D_i$ be a component of $C$, that is,
$c_i>0$. Since $D_i$ is smooth over the base, it is horizontal. We
follow its strict transforms through the relative MMP, and denote
them again by $D_i$ until $D_i$ is contracted; at the beginning, the
central fiber of $D_i$ is the component $D_{i,0}$ of $C_0$.

First consider a contraction morphism $\varpi:V\to W$ extending a
central divisorial step, with $D_i\subset V$, and let $V_0$ and $W_0$
be the central fibers. Since $\varpi$ is a proper morphism over the
base, we have
\[
 \varpi(D_i)\cap W_0=\varpi(D_i\cap V_0)
\]
as sets. If the central transform of $D_i$ is contracted, this set
has dimension at most $n-2$. If $\varpi(D_i)$ were a horizontal
divisor, the base coordinate would cut it in a set of dimension at
least $n-1$, which is nonempty by properness. This is impossible, and
hence $D_i$ is contracted also on the total space. If the central
transform of $D_i$ is not contracted, the same equality identifies
the central fiber of $\varpi(D_i)$ with the image of that transform.

Next consider a relative flip $V\to T\leftarrow V^+$, and let $D_i^+$
be the strict transform of $D_i$. Near a point where a central
contraction is an isomorphism, the relative contraction has finite
fibers by properness and the upper semicontinuity of fiber
dimension; since it is finite and bimeromorphic onto a normal space,
it is an isomorphism there. Consequently, $D_i^+\cap V^+_0$ coincides
with the central strict transform outside the exceptional locus of
the small central contraction $V^+_0\to T_0$. This exceptional locus
contains no further components. In fact, since $D_i^+$ is horizontal,
the base coordinate cuts a nonzero principal divisor on it; hence
every irreducible component of $D_i^+\cap V^+_0$ has dimension $n-1$,
whereas the exceptional locus has dimension at most $n-2$. Thus the
support of $D_i^+\cap V^+_0$ is the central strict transform. Since
only supports are used here, the intermediate divisors need not be
smooth over the base.

Following $D_i$ through the finite sequence, we see that either it is
contracted during the MMP or the central fiber of its transform on
$X'$ is the transform of $D_{i,0}$. In the latter case, $h_0$
contracts this central transform by
Lemma~\ref{lem:central-exceptionality}, and the same dimension
argument, applied to $h$, shows that the transform of $D_i$ is
$h$-exceptional. Thus $\X\dashrightarrow Y$ contracts every component
of $C$.

Moreover, every $\mu$-exceptional divisor is $f$-exceptional, since
$\X\dashrightarrow Y$ is a bimeromorphic contraction; otherwise its
image would be a divisor on $Y$ contracted by the inverse map
$Y\dashrightarrow\X$. Since every component of $\mu^*C$ is either the
strict transform of a component of $C$ or a $\mu$-exceptional
divisor, $\mu^*C$ is $f$-exceptional.

\medskip\noindent
\emph{Step 3: Volumes.}
Put \(E_\alpha=E'+\mu^*C\). Then \eqref{eq:restored-model} follows
from the preceding steps. We take a further log resolution which also
includes the supports of $\mu^*B$, $E_\alpha$ and the relative
discrepancy divisor. By generic smoothness and properness, there is a
closed discrete set of parameters outside which the resolution, the
divisor strata and the modifications of the fibers have the stated
properties. We add to this set the finitely many parameters at which
the fiber dimension of the image of a horizontal exceptional
component exceeds its generic value, which is at most $n-2$, together
with the images of the vertical components. At the remaining
parameters, each component of the restriction of an exceptional
divisor is exceptional on the fiber. After shrinking the disk, we may
assume that $0$ is the only possible exceptional parameter.

For $t\ne0$ the class $[P_t]=f_t^*[\gamma_t]$ is nef and big. Adding
the effective divisor $E_{\alpha,t}$, we see that $\mu_t^*\alpha_t$,
and hence $\alpha_t$, is big. Since $E_{\alpha,t}$ is
$f_t$-exceptional, the divisorial Zariski decomposition
\cite[Theorem~3.18 and Proposition~3.20]{BouZ} shows that the addition
of $E_{\alpha,t}$ does not change the volume of $[P_t]$. This volume
is the top self-intersection number of $[P_t]$. Therefore, by the
bimeromorphic invariance of the volume,
\[
 \vol(\alpha_t)=\int_{Z_t}P_t^n=\int_{Y_t}\gamma^n
 \qquad(t\ne0).
\]
The same argument, applied to the central model
\eqref{eq:central-reduced-model} together with
Lemma~\ref{lem:central-exceptionality}, gives
$
 \vol(\alpha_0)=\int_{Y_0}\gamma_0^n.
$
The fiber integrals of the closed form \(\gamma^n\) are constant
in the proper flat family \(Y\to\Delta\); this follows by
integration over the continuously varying fiber cycles, or by
Stokes' theorem for this family of cycles. The fibers are normal
and hence reduced, so that their fundamental cycles have
multiplicity one. Since $[\gamma]|_{Y_0}=\gamma_0$, the common value
is $\int_{Y_0}\gamma_0^n>0$. Thus $\vol(\alpha_t)$ is constant on the
smaller disk, and every $\alpha_t$ is big.
\end{proof}

For the analytic part of the argument we denote the relative
discrepancy divisor by $E_\mu$ and define a signed divisor $D$ by
\[
 E_\mu:=K_{Z/\Delta}-\mu^*K_{\X/\Delta},\qquad
 D:=E_\alpha+E_\mu-\mu^*B.
\]
Then
\begin{equation}\label{eq:signed-model-identity}
 c_1(K_{Z/\Delta})+[\mu^*\beta]+\varepsilon[\mu^*\omega]=[P]+[D].
\end{equation}
We write $D=D_+-D_-$, where $D_+$ and $D_-$ are the coefficientwise
positive and negative parts of $D$; thus $D_+,D_-\ge0$ have no common
components. The divisor $E_\mu$ is effective and $\mu$-exceptional,
and hence $f$-exceptional. Since $E_\alpha$ is also $f$-exceptional,
$D_+$ is $f$-exceptional; $D_-$ need not be $f$-exceptional.

Fix $\varepsilon>0$ and the model of
Proposition~\ref{prop:local-log-model}. We choose a closed discrete
set $\Sigma\subset\Delta$ as in its proof, and put
$\Delta^\circ:=\Delta\setminus\Sigma$ and
$Z^\circ:=q^{-1}(\Delta^\circ)$, so that the following conditions are
satisfied:
\begin{enumerate}
\item $q|_{Z^\circ}:Z^\circ\to\Delta^\circ$ is a smooth proper
submersion;
\item for every $t\in\Delta^\circ$, the maps
$\mu_t:Z_t\to X_t$ and $f_t:Z_t\to Y_t$ are modifications;
\item the union of the supports of $\mu^*B$, $E_\alpha$ and $E_\mu$
is relative SNC over $\Delta^\circ$;
\item for every $t\in\Delta^\circ$, each prime component of
$(E_\alpha+E_\mu)|_{Z_t}$ is $f_t$-exceptional.
\end{enumerate}
We call a parameter $t$ \emph{good} if $t\in\Delta^\circ$, and we
call the corresponding resolved fiber $Z_t=q^{-1}(t)$ a \emph{good
fiber}. These notions refer to the fixed model and to the chosen set
$\Sigma$, which may depend on $\varepsilon$. By a good fiber $X_t$ of
the original family we mean a fiber whose parameter is good in this
sense; we note that the original fibers $X_t$ are smooth for all
$t\in\Delta$, including those in $\Sigma$.

On every good fiber the coefficients of \(E_\mu-\mu^*B\) are greater
than $-1$ by the klt condition, and this bound remains valid for $D$,
since \(E_\alpha\) is effective. Furthermore, every component $F$ of
$D_+$ is orthogonal to $P$ on a good fiber. In fact, $F_t$ is
$f_t$-exceptional, and hence, by the projection formula,
\[
 \int_{F_t}P_t^{n-1}
 =\int_{(f_t)_*F_t}\gamma_t^{n-1}=0.
\]
These are the properties used in Section~\ref{sec:signed-analysis}.
The excluded parameters, at which the resolved fiber need not be
smooth, will be treated in Section~\ref{sec:log-descent} by descent
to the original smooth family.

\subsection{Continuation on a fixed disk}

The disk in Proposition~\ref{prop:local-log-model} may depend on
$\varepsilon$. In order to let $\varepsilon\to0$, we need the
constancy of the volume on a disk independent of $\varepsilon$. In
the following lemma we extend the local constancy from the big locus
to the whole connected disk. The proof rests on the fact that a limit
of parameters with the same positive volume again has positive
volume; for this we use the compactness of positive currents and
Boucksom's upper semicontinuity theorem \cite[Proposition~4.17]{BouV}.

\begin{lemma}\label{lem:clopen-volume}
Let $p:\X\to\Delta$ be a smooth proper K\"ahler family over a
connected disk, and let $\xi$ be a smooth real closed $(1,1)$-form
on $\X$. Suppose that, whenever \([\xi_t]\) is big, there is a
neighborhood of \(t\) in the disk on which all fiber classes are big
and have the same volume. Then, if one fiber class is big, every
fiber class is big and the volume is constant on the disk.
\end{lemma}

\begin{proof}
Let \(\Omega\) be the set of parameters $t$ for which $[\xi_t]$ is
big, and let \(\Omega_0\) be the connected component of \(\Omega\)
containing a given big parameter. By hypothesis \(\Omega_0\) is open,
and the volume is constant on \(\Omega_0\), equal to some \(v>0\).
Let \(t_j\in\Omega_0\) be a sequence converging to a parameter
\(t_\infty\), and choose positive currents $T_j\in[\xi_{t_j}]$. For a
fixed K\"ahler form $\omega$ on the total space, the masses
\[
 \int_{X_{t_j}}T_j\wedge\omega_{t_j}^{n-1}
 =\int_{X_{t_j}}\xi_{t_j}\wedge\omega_{t_j}^{n-1}
\]
are bounded for $t_j$ near $t_\infty$. Identifying nearby fibers by
an Ehresmann trivialization, we find a subsequence which converges
weakly to a closed current $T_{t_\infty}$. By means of local
holomorphic submersion charts we see that the limit is positive and
of type $(1,1)$ with respect to the complex structure of
$X_{t_\infty}$. Its de Rham class is the restriction of $[\xi]$.
Hence, by the $\partial\bar\partial$-lemma on the compact K\"ahler
manifold $X_{t_\infty}$, $T_{t_\infty}$ lies in the Bott--Chern class
$[\xi_{t_\infty}]$, which is therefore pseudoeffective.

By the upper semicontinuity of the volume under K\"ahler deformations
\cite[Proposition~4.17]{BouV}, we have
\[
 \vol([\xi_{t_\infty}])\geq
 \limsup_j\vol([\xi_{t_j}])=v>0,
\]
and a class with positive volume on a compact K\"ahler manifold is
big. By hypothesis, $t_\infty$ then has a connected open
neighborhood contained in \(\Omega\). This neighborhood meets
\(\Omega_0\), so that \(t_\infty\in\Omega_0\). Thus \(\Omega_0\) is
open and closed in the connected disk, and hence coincides with it.
\end{proof}

\begin{corollary}
\label{cor:common-domain}
In the canonicalized family above, put
\(\mathfrak a_t=c_1(K_{X_t}+B_t)+[\beta_t]\). If \(\mathfrak a_0\) is
pseudoeffective, then for every \(\varepsilon>0\) the classes
\(\mathfrak a_t+\varepsilon[\omega_t]\) are big on the original
connected disk and have constant volume there. Consequently,
\(\mathfrak a_t\) is pseudoeffective for every \(t\).
\end{corollary}

\begin{proof}
Fix \(\varepsilon>0\). Since $\mathfrak a_0$ is pseudoeffective,
$\mathfrak a_0+\varepsilon[\omega_0]$ is big. At every parameter where
the perturbed class is big, Proposition~\ref{prop:local-log-model}
gives a neighborhood on which the classes are big and have constant
volume. Hence, by Lemma~\ref{lem:clopen-volume}, the classes are big
and have constant volume on the whole disk. Letting
\(\varepsilon\to0\) on each fixed fiber, we see that $\mathfrak a_t$ is
pseudoeffective, since the pseudoeffective cone is closed.
\end{proof}

Thus the local models may depend both on the base parameter and on
$\varepsilon$, but the disk in Corollary~\ref{cor:common-domain} is
independent of $\varepsilon$, and the analytic limit can be taken on
this fixed disk.

In the extension problem, $\beta$ is the curvature of a smooth
semipositive metric on the rational line bundle $M$. If there is a
nonzero section of $m(K_{X_0}+B_0+M_0)$ in an admissible degree, then
$\mathfrak a_0$ is pseudoeffective, as required in
Corollary~\ref{cor:common-domain}.

\begin{proof}[Proof of Theorem~\ref{thm:real-log-volume}]
First we assume that every $b_i<1$. Let
$\mathfrak a_t:=c_1(K_{X_t}+B_t)+[\beta_t]$ be the real adjoint class
in the theorem, and let $S\subset\Delta$ be the set of parameters for
which it is pseudoeffective. We first show that $S$ is open and that
the volume is locally constant on $S$. Fix $t_0\in S$. By
Lemma~\ref{lem:canonicalization} for real coefficients, there is,
over a small disk $U$ around $t_0$, a relative crepant modification
$\tau:\widetilde\X\to p^{-1}(U)$ with canonical fiber pairs.
With the twist $\tau^*\beta$, the adjoint classes of this family are
$\widetilde{\mathfrak a}_t=\tau_t^*\mathfrak a_t$. We choose a
K\"ahler form $\widetilde\omega$ on $\widetilde\X$.

We apply Corollary~\ref{cor:common-domain} with center $t_0$. Then
every $\widetilde{\mathfrak a}_t$ is pseudoeffective on $U$, and, for
each $\varepsilon>0$, the classes
$\widetilde{\mathfrak a}_t+\varepsilon[\widetilde\omega_t]$ are big
and have constant volume on the same disk $U$. Hence, by the
continuity of the volume along decreasing K\"ahler perturbations
\cite{BouV},
\[
 \vol(\widetilde{\mathfrak a}_t)
 =\lim_{\varepsilon\downarrow0}
    \vol(\widetilde{\mathfrak a}_t+\varepsilon[\widetilde\omega_t])
 =\vol(\widetilde{\mathfrak a}_{t_0}).
\]
By the bimeromorphic invariance of pseudoeffectivity and of the
volume, we have $U\subset S$ and
$\vol(\mathfrak a_t)=\vol(\mathfrak a_{t_0})$ for every $t\in U$.

The compactness argument for positive currents in the proof of
Lemma~\ref{lem:clopen-volume} shows also that $S$ is closed in
$\Delta$. Since $\Delta$ is connected, $S$ is either empty or all of
$\Delta$. In the first case the volume vanishes by convention; in the
second case it is constant by the local constancy just proved. This
proves both assertions for coefficients strictly less than one.

Now let $0\le b_i\le1$. If no $\mathfrak a_t$ is pseudoeffective,
both assertions hold, all volumes being zero by convention. Otherwise
we fix $t_0$ such that $\mathfrak a_{t_0}$ is pseudoeffective, and
let $t_1\in\Delta$ be arbitrary. We choose a relatively compact
subdisk $U\Subset\Delta$ containing $t_0$ and $t_1$, and fix a
K\"ahler form $\omega$ on $\X$. Smooth Hermitian metrics on the
bundles $\OO_\X(D_i)$ give a smooth closed representative $\theta_B$
of $[B]$. Since $p^{-1}(\overline U)$ is compact, we can choose
$C>0$ such that
\[
 \eta:=C\omega-\theta_B>0
\]
on this compact set. In particular, $\eta$ is a K\"ahler form on
$p^{-1}(U)$. For $0<\lambda<1$ we put
\[
 B_\lambda=(1-\lambda)B,\qquad
 \beta_\lambda=\beta+\lambda C\omega.
\]
The coefficients of $B_\lambda$ are strictly less than one, and
$\beta_\lambda$ is smooth, closed and semipositive. The corresponding
adjoint classes are
\[
 \mathfrak a_{t,\lambda}
 =c_1(K_{X_t}+B_{\lambda,t})+[\beta_{\lambda,t}]
 =\mathfrak a_t+\lambda[\eta_t].
\]
At $t_0$ this class is the sum of a pseudoeffective class and a
K\"ahler class, and hence it is big. Applying the case already proved
to $p^{-1}(U)\to U$, with boundary $B_\lambda$ and twist
$\beta_\lambda$, we see that $\mathfrak a_{t,\lambda}$ is
pseudoeffective for every $t\in U$ and that its volume is independent
of $t$. Letting $\lambda\downarrow0$ on each fixed fiber, we see that
$\mathfrak a_t$ is pseudoeffective, since the pseudoeffective cone is
closed. By the continuity of the volume along decreasing K\"ahler
perturbations \cite{BouV}, we have
\[
 \vol(\mathfrak a_{t_1})
 =\lim_{\lambda\downarrow0}\vol(\mathfrak a_{t_1,\lambda})
 =\lim_{\lambda\downarrow0}\vol(\mathfrak a_{t_0,\lambda})
 =\vol(\mathfrak a_{t_0}).
\]
Since $t_1$ is arbitrary, both assertions hold on the whole disk.
The last assertion of the theorem follows, since a class is big if
and only if its volume is positive.
\end{proof}

\begin{remark*}
Theorem~\ref{thm:real-log-volume} admits coefficients equal to one,
but it concerns only real adjoint classes. The extension of
individual sections in this case requires additional hypotheses,
which we treat in Section~\ref{sec:applications}.
\end{remark*}

\begin{corollary}
\label{cor:boundary-volume-function}
Fix relative divisors $D_1,\ldots,D_r$ whose union has relative
simple normal crossings, and a smooth closed semipositive form
$\beta$. For each $t$ define
\[
 \mathcal P_t=
 \left\{b=(b_1,\ldots,b_r)\in[0,1]^r:
   c_1\left(K_{X_t}+\sum_i b_iD_{i,t}\right)+[\beta_t]
   \text{ is pseudoeffective}\right\}.
\]
Then $\mathcal P_t$ is a closed convex subset of $[0,1]^r$, possibly
empty, and it is independent of $t$. Moreover, the function
\[
 b\longmapsto
 \vol\left(c_1\left(K_{X_t}+\sum_i b_iD_{i,t}\right)
               +[\beta_t]\right)
\]
on $[0,1]^r$ is independent of $t$.
\end{corollary}
\begin{proof}
Applying Theorem~\ref{thm:real-log-volume} to each coefficient vector
$b$ separately, we obtain both invariance statements. On a fixed
fiber the map from $b$ to the adjoint class is affine and continuous.
Hence the inverse image of the closed convex pseudoeffective cone,
intersected with $[0,1]^r$, is closed and convex.
\end{proof}

\section{Variation of metrics with zeros and poles}
\label{sec:signed-analysis}

In this section we shall prove that the solutions of the
Monge--Amp\`ere equations on the fibers of a relative model define a
potential which is jointly psh with respect to the background form.
The boundary gives poles of the density, and the exceptional divisors
give zeros. We describe the density by a signed divisor, whose
positive and negative parts correspond to the zeros and the poles. In
the curvature estimate the poles have a favorable sign, while the
zeros give an error term, which vanishes by the orthogonality proved
in Section~\ref{sec:relative-log-geometry}. We follow the argument
for effective divisors in \cite[Sections~5--7]{LPWZ} and indicate the
changes needed in the presence of poles.

Throughout this section the perturbation $\varepsilon>0$ used in the
construction of the model is fixed, and constants may depend on the
model and hence on $\varepsilon$. We introduce three further
parameters: $\delta$ makes the semipositive background K\"ahler,
$a_1$ smooths the zeros, and $a_2$ smooths the poles. The estimates
used in passing to the limit are uniform in these three parameters.
We let them tend to zero in the order $a_1\downarrow0$, then
$a_2\downarrow0$, and finally $\delta\downarrow0$.

We shall use the regularity theory for the Perron problem of the
homogeneous Monge--Amp\`ere equation \cite[Corollary~1.3]{CTW17},
\cite[Theorems~1.1 and~2.1]{CTW18}, the positivity of local weighted
Bergman kernels \cite[Theorem~1.1]{Ber06}, and the \(L^1\)
asymptotics on the fibers of Bergman kernels for bounded densities
\cite[Theorem~2.1 and Corollary~2.2]{BB17}, applied to proper K\"ahler
families as in \cite[Sections~5--6]{LPWZ}. The hypothesis on the
boundary data in the Perron problem will be verified below. For fixed
regularization parameters, the Bergman weight is always a bounded psh
potential.

\subsection{The model equation and uniform estimates}

For the fixed model in Proposition~\ref{prop:local-log-model}, recall
that $q=p\circ\mu=g\circ f:Z\to\Delta$ is the map to the base.
We keep the excluded set $\Sigma$ chosen in
Section~\ref{sec:relative-log-geometry}, and recall the notation
\[
 \Delta^\circ:=\Delta\setminus\Sigma,
 \qquad Z^\circ:=q^{-1}(\Delta^\circ).
\]
Thus $\Delta^\circ$ parametrizes the good fibers, and the restricted
family $q|_{Z^\circ}:Z^\circ\to\Delta^\circ$ is smooth and proper.
Its fibers are connected, since each $\mu_t:Z_t\to X_t$ is a
modification and $X_t$ is connected.
Throughout this section we denote this restricted family simply by
$q:Z\to\Delta^\circ$.

The following proposition is an extension of
\cite[Proposition~7.1]{LPWZ} to densities with divisorial zeros and
with poles of order less than one. The proof uses the same comparison
with Perron envelopes, together with additional estimates for the
pole factors.

\begin{proposition}
\label{prop:signed-model}
Let $q:Z\to\Delta^\circ$ be a smooth proper K\"ahler family with
connected fibers of dimension $n$, let $\eta$ be a K\"ahler form on $Z$, and
let $P\ge0$ be a smooth closed real $(1,1)$-form whose restriction
$P_t$ has a big class on every fiber.
Let $E,C\ge0$ be real divisors with relative SNC union, and write
\[
 D=E-C=\sum_\ell\lambda_\ell Q_\ell
\]
with distinct prime divisors $Q_\ell$; assume that $\lambda_\ell>-1$
for every $\ell$. Choose smooth Hermitian metrics $H_\ell$ on
$\mathcal O_Z(Q_\ell)$, let $s_{Q_\ell}$ be the canonical sections,
and put
\[
 \Theta_D:=\sum_\ell\lambda_\ell\Theta(H_\ell),\qquad
 |s_D|^2:=\prod_\ell |s_{Q_\ell}|_{H_\ell}^{2\lambda_\ell};
\]
these expressions make sense for real coefficients, and $|s_D|^2$ is
defined off $\Supp D$. Let $\theta_\eta$ be the curvature of the
metric on $K_{Z/\Delta^\circ}$ induced by the \emph{same} vertical
volume forms $\eta_t^n$ used in the equation below. Suppose that
there are a smooth real function $h$ and a smooth closed semipositive
real $(1,1)$-form $b$ such that
\begin{equation}\label{eq:analytic-model-identity}
 \theta_\eta+b=P+\Theta_D+\ddc h,\qquad b\ge0,
 \qquad \int_{E_{i,t}}P_t^{n-1}=0
\end{equation}
for every prime component $E_i$ of $E$ and every $t\in\Delta^\circ$.
Put
\[
 f:=e^{-h}|s_D|^2.
\]
Then on each fiber there is a unique bounded function
$u_t\in\PSH(Z_t,P_t)$ satisfying
\begin{equation}\label{eq:analytic-model-equation}
 (P_t+\ddc u_t)^n=e^{u_t}f_t\eta_t^n,
\end{equation}
and the function $u(x):=u_{q(x)}(x)$ is $P$-psh on $Z$.
\end{proposition}

We compare the solutions on the fibers with the Perron envelope of the
homogeneous Monge--Amp\`ere equation over a disk in the base. First
we regularize the equation and obtain uniform bounds
(Lemma~\ref{lem:signed-regularization-details}). Next, by a local
inequality for Bergman kernels and by Green's formula, we estimate the
positive part of the logarithm of the ratio of the Monge--Amp\`ere
measure of the envelope to the prescribed measure
(Lemmas~\ref{lem:signed-local-current-details}
and~\ref{lem:signed-green-details}). The error terms vanish when the
regularizations are removed (Lemma~\ref{lem:signed-error-details}).
Finally, the comparison principle shows that the solutions on the
fibers lie below the limiting envelope, and this gives the submean
inequalities needed for the joint positivity
(Lemma~\ref{lem:signed-limit-details}).

Fix a disk $\mathbb D$ in the base with
$\overline{\mathbb D}\subset\Delta^\circ$; where regularity at the
boundary is needed, we work over a slightly larger closed disk. We
define the positive and negative parts of $D$ by
\[
 D_+:=\sum_\ell\max\{\lambda_\ell,0\}Q_\ell,\qquad
 D_-:=\sum_\ell\max\{-\lambda_\ell,0\}Q_\ell.
\]
If we replace $E$ and $C$ by $D_+$ and $D_-$, their common components
cancel, and \eqref{eq:analytic-model-identity} and the orthogonality
assumption remain valid, since $D_+\le E$. We may therefore write
\[
 E=\sum_i d_iE_i,\qquad C=\sum_j c_jC_j,
 \qquad d_i>0,\quad 0<c_j<1,
\]
where the zero divisor $E$ and the pole divisor $C$ have no common
components. We denote by $s_{E_i},s_{C_j}$ their canonical sections,
with the smooth metrics chosen above, and we suppress the metrics in
the norms. In particular,
\[
 f=e^{-h}\prod_i|s_{E_i}|^{2d_i}\prod_j|s_{C_j}|^{-2c_j}.
\]
The $L^p$ norms on the fibers are taken with respect to $\eta_t^n$,
unless a coordinate Lebesgue measure is specified. Constants $c,C>0$
may change from line to line, and a constant is called uniform if it
does not depend on $\delta,a_1,a_2$ or on $t\in\overline{\mathbb D}$.

The next lemma is an adaptation of \cite[Lemma~5.5]{LPWZ} to the
density $f$. In the presence of poles the regularized densities need
not be uniformly bounded, but they are uniformly bounded in $L^p$ for
some $p>1$. This weaker bound suffices for the estimates and for the
two successive limits of the densities.

\begin{lemma}
\label{lem:signed-regularization-details}
For $0<\delta,a_1,a_2\le1$, put
\[
 P_\delta:=P+\delta\eta,\qquad
 f_{a_1,a_2}:=e^{-h}
     \prod_i(|s_{E_i}|^2+a_1)^{d_i}
     \prod_j(|s_{C_j}|^2+a_2)^{-c_j}.
\]
We use this formula also when $a_1=0$ or $a_2=0$, interpreted almost
everywhere; thus $f_{0,0}=f$. Let $u_{\delta,a_1,a_2,t}$ denote
the smooth $P_{\delta,t}$-psh solution of
\[
 (P_{\delta,t}+\ddc u_{\delta,a_1,a_2,t})^n
   =e^{u_{\delta,a_1,a_2,t}}f_{a_1,a_2,t}\eta_t^n
\]
and let $u_{\delta,a_1,a_2}$ be the function on
$q^{-1}(\overline{\mathbb D})$ whose restrictions to the fibers are
these solutions. Then these solutions depend smoothly on $t$ and
satisfy
\begin{equation}\label{eq:model-regularization-bound}
 \sup_{q^{-1}(\overline{\mathbb D})}|u_{\delta,a_1,a_2}|\le C_{\mathbb D}
\end{equation}
with $C_{\mathbb D}$ independent of $\delta,a_1,a_2$.
For fixed $\delta$, letting first $a_1\downarrow0$ and then
$a_2\downarrow0$, we obtain uniform limits on
$q^{-1}(\overline{\mathbb D})$ and a continuous function $u_\delta$.
As $\delta\downarrow0$, the functions $u_\delta$ decrease to a
bounded upper semicontinuous function $u$ whose restrictions to the
fibers satisfy \eqref{eq:analytic-model-equation}.
\end{lemma}

\begin{proof}
\emph{Step 1: Smooth solutions and the densities.}
If all three regularization parameters are positive, the background
form is K\"ahler and the density is smooth and positive. Therefore
the exponential Monge--Amp\`ere equation has a unique smooth solution
\cite{Aubin,Yau}. Moving the exponential to the left-hand side, we
obtain a linearized operator with a negative zeroth-order term. By the
maximum principle and elliptic theory this operator is invertible,
and hence, by the implicit function theorem, the solution depends
smoothly on $t$.

Fix $p>1$ with $pc_j<1$ for every $j$. In a relative SNC chart we
choose vertical coordinates $z=(z_1,\ldots,z_n)$ such that, after
relabeling the pole components meeting the chart, they are given by
$z_j=0$. The pole factor is bounded by a fixed multiple of
$\prod_j|z_j|^{-2c_j}$, and the zero factors are bounded. Covering the
family by finitely many charts, we obtain
\[
 \|f_{a_1,a_2,t}\|_{L^p(\eta_t^n)}\le C,
 \qquad
 0<c\le\int_{Z_t}f_{a_1,a_2,t}\eta_t^n\le C.
\]
For the lower bound we use $f_{a_1,a_2}\ge f_{0,1}$; the fiber
integrals of $f_{0,1}$ are positive and continuous on
$\overline{\mathbb D}$.

We shall also need the convergence of the densities uniformly in the
parameter $t$. With $a_2>0$ fixed, $f_{a_1,a_2,t}\to f_{0,a_2,t}$ in
$L^1(Z_t,\eta_t^n)$ as $a_1\downarrow0$, and then
$f_{0,a_2,t}\to f_t$ in the same space as $a_2\downarrow0$.
Both $L^1$ errors tend to zero uniformly for $t\in\overline{\mathbb D}$.
In fact, in relative SNC charts the common integrable bound makes the
contribution near the pole divisors uniformly small, and outside a
fixed neighborhood of these divisors the densities converge
uniformly on compact sets.

\medskip\noindent
\emph{Step 2: Uniform bounds.}
For fixed $\delta,a_1,a_2,t$ we put
$w:=u_{\delta,a_1,a_2,t}-\sup_{Z_t}u_{\delta,a_1,a_2,t}$, so that
$\sup_{Z_t}w=0$. Since the forms $P_{\delta,t}$ are bounded by a fixed
multiple of $\eta_t$, the uniform Skoda integrability gives a bound
for $\int_{Z_t}|w|^\ell\eta_t^n$ for every real $\ell\ge1$. Hence, by
H\"older's inequality and the $L^p$ bound in Step~1,
$\int_{Z_t}|w|f_{a_1,a_2,t}\eta_t^n$ is bounded. We apply Jensen's
inequality to the probability measure obtained by dividing
$f_{a_1,a_2,t}\eta_t^n$ by its total mass. Combined with the bounds
for the mass in Step~1 and with $w\le0$, this gives
\[
 0<c\le J:=\int_{Z_t}e^w f_{a_1,a_2,t}\eta_t^n\le C.
\]
The volumes $V_\delta=\int_{Z_t}P_{\delta,t}^n$ are independent of
$t$, since the forms are closed, and they are bounded from above and
away from zero uniformly in $\delta$, since $P_t$ is semipositive and
big. Integrating the equation, we have
$J\exp\bigl(\sup_{Z_t}u_{\delta,a_1,a_2,t}\bigr)=V_\delta$, and hence
$\bigl|\sup_{Z_t}u_{\delta,a_1,a_2,t}\bigr|$ is uniformly bounded.
After division by $V_\delta$, the equation for $w$ has the
probability density $e^w f_{a_1,a_2,t}/J$, whose $L^p$ norm is
uniformly bounded. Therefore $\|w\|_\infty$ is bounded by the uniform
Monge--Amp\`ere estimate used in \cite[Sec.~5]{LPWZ}, or by
\cite[Theorem~1.9]{DGG} with $V_{P_{\delta,t}}=0$. This proves
\eqref{eq:model-regularization-bound}.

\medskip\noindent
\emph{Step 3: Removing the density regularizations.}
We now derive the uniform convergence of the solutions from the
$L^1$ convergence of the densities. Let $f_1\ge f_2$ be two
regularized densities, with solutions $u_1,u_2$ in
$\PSH(Z_t,P_{\delta,t})$. By the comparison principle, $u_1\le u_2$.
We define the probability measure
$d\mu_1:=V_\delta^{-1}e^{u_1}f_1\eta_t^n$. By the inequality
$x\le e^x-1$ for $x\ge0$, we have
\begin{align*}
 \int_{Z_t}(u_2-u_1)\,d\mu_1
 &\le V_\delta^{-1}\int_{Z_t}(e^{u_2}-e^{u_1})f_1\eta_t^n\\
 &=V_\delta^{-1}\int_{Z_t}e^{u_2}(f_1-f_2)\eta_t^n
 \le C\|f_1-f_2\|_{L^1}.
\end{align*}
In the equality we have used that both Monge--Amp\`ere measures have
total mass $V_\delta$. By the uniform Skoda integrability and the
$L^p$ bound, the finite moments required in the stability theorem
\cite[Theorem~2.6]{GL25}, recalled in \cite[Theorem~5.1]{LPWZ}, are
bounded. In order to apply that theorem, we subtract $\sup u_1$ from
both potentials. This normalizes the solution without changing the
difference, and the comparison function remains uniformly bounded by
Step~2. The stability estimate then gives, for uniform
constants $C,\gamma_{\mathrm{st}}>0$,
\[
 \|u_2-u_1\|_\infty
       \le C\|f_1-f_2\|_{L^1}^{\gamma_{\mathrm{st}}}.
\]
As $a_1\downarrow0$, the densities decrease and the solutions
increase. After taking this limit, we let $a_2\downarrow0$; then the
densities increase and the solutions decrease. In each step the
uniform $L^1$ convergence in Step~1 and the stability estimate give
uniform convergence on $q^{-1}(\overline{\mathbb D})$. Hence the
resulting function $u_\delta$ is continuous. By the continuity
theorem of Bedford--Taylor for bounded potentials, its restrictions
to the fibers satisfy the equation with density $f$.

\medskip\noindent
\emph{Step 4: Removing the background perturbation.}
For $\delta'>\delta$ the background form increases by the positive
form $(\delta'-\delta)\eta$. Hence, on each fiber,
\[
 (P_{\delta',t}+\ddc u_{\delta,t})^n
 \ge e^{u_{\delta,t}}f_t\eta_t^n.
\]
Thus $u_{\delta,t}$ is a subsolution of the equation satisfied by
$u_{\delta',t}$, and the comparison principle gives
$u_\delta\le u_{\delta'}$. In particular, $u_\delta$ decreases as
$\delta\downarrow0$. By the uniform bound
\eqref{eq:model-regularization-bound}, the limit $u$ is bounded.
Since each $u_\delta$ is continuous, this decreasing limit is upper
semicontinuous.

In order to pass to the limit in the equation, we choose smooth local
potentials $A,R$ with $\ddc A=P$ and $\ddc R=\eta$. Adding a constant
to $R$ on a smaller chart, we may assume that $R\ge0$. The functions
\[
 v_\delta:=u_\delta+A+\delta R
\]
are psh on each fiber and decrease to $u+A$, since both $u_\delta$
and $\delta R$ decrease. They are locally uniformly bounded. Hence,
by the convergence theorem of Bedford--Taylor \cite[Theorem~2.1]{BT82},
$(P_{\delta,t}+\ddc u_{\delta,t})^n$ converges weakly to
$(P_t+\ddc u_t)^n$. On the right-hand side we can apply the dominated
convergence theorem, because of the bound
$0\le e^{u_{\delta,t}}f_t\le e^{C_{\mathbb D}}f_t$ and the
integrability of $f_t$. Thus each $u_t$ satisfies
\eqref{eq:analytic-model-equation}. The comparison principle for this
exponential equation gives the uniqueness among bounded $P_t$-psh
solutions.
\end{proof}

\subsection{Perron envelopes and a local curvature inequality}

The above construction gives positivity in the fiber directions. In
order to prove positivity in the directions involving the base
variable, we compare the solutions with a Perron envelope which is
psh, relative to the background form, on the total space. In the next
three lemmas we fix $0<\delta,a_1,a_2\le1$, and we write
$\mathbb D=\{|t-t_*|<r_*\}$.

Let $\mathcal F$ be the family of all functions
$v\in\PSH(q^{-1}(\mathbb D),P_\delta)$ whose boundary values do not
exceed the given solutions on the fibers, in the sense that
$
 \limsup_{q^{-1}(\mathbb D)\ni x\to\xi}v(x)
 \le u_{\delta,a_1,a_2}(\xi)
$ for every $\xi\in q^{-1}(\partial \mathbb D).$
The Perron envelope is defined by
$
 U=U_{\delta,a_1,a_2}:=
       \left(\sup_{v\in\mathcal F}v\right)^*.
$
By construction $U$ is jointly $P_\delta$-psh. The following
regularity argument shows that $U$ attains the prescribed boundary
values, that is, $U=u_{\delta,a_1,a_2}$ on $q^{-1}(\partial \mathbb D)$.

In order to apply the regularity theorem, we first construct a smooth
strict subsolution with the prescribed boundary values. The vertical
block of $P_\delta+\ddc u_{\delta,a_1,a_2}$ is strictly positive. On
the compact closure of $q^{-1}(\mathbb D)$ its smallest eigenvalue is
bounded from below by a positive constant, and the other coefficients
are bounded. Hence, if we add a large positive term in the base
direction, the whole form becomes strictly positive. More precisely,
by the Schur complement criterion,
\[
 \underline U=u_{\delta,a_1,a_2}
                 +\Lambda\bigl(|t-t_*|^2-r_*^2\bigr)
\]
satisfies $P_\delta+\ddc\underline U>0$ for sufficiently large
$\Lambda$. The added term vanishes on $q^{-1}(\partial \mathbb D)$, so
that $\underline U$ has the required boundary values. Moreover, this
boundary is Levi-flat and satisfies the boundary hypothesis of
\cite[Theorem~1.1]{CTW18}. Put $T:=P_\delta+\ddc U$ on the total
space and $T_t:=P_{\delta,t}+\ddc U_t$ on each fiber. By the
regularity theorem for the Perron problem, we have
\[
 U\in C^{1,1},\qquad T\ge0,
 \qquad T^{n+1}=0,
 \qquad U|_{q^{-1}(\partial \mathbb D)}=u_{\delta,a_1,a_2}.
\]
In particular, the coefficients of $T$ are locally bounded.
The constant $\Lambda$, and hence the $C^{1,1}$ norm of $U$, may
depend on the three regularization parameters. We shall use uniformly
only the bound for the supremum norm of $U$:
\begin{equation}\label{eq:signed-perron-uniform-bound}
 |U|\le C'_{\mathbb D}.
\end{equation}
This is the bound for Perron envelopes in \cite[Lemma~5.6]{LPWZ}.
Since it uses only the uniform bound for the boundary values and the
fixed forms $P,\eta$, it applies without change to the present
regularized boundary data.

The following local inequality is an extension of
\cite[Lemma~6.1]{LPWZ} to $D=E-C$. The pole divisor $C$ contributes
with a favorable sign. However, it may make the auxiliary psh
function unbounded from below, and we therefore justify the
calculation with currents by truncating that function.

\begin{lemma}
\label{lem:signed-local-current-details}
We work in product coordinates $(z,t)$, with
$z=(z_1,\ldots,z_n)$ along the fibers. Choose a smooth local potential
$a_\delta$ such that $\ddc a_\delta=P_\delta$, and put
\[
 v:=a_\delta+U,\qquad T=\ddc v,\qquad S:=T^n.
\]
Thus $v$ is a locally bounded psh function, and $S$ is an $(n,n)$
current on the $(n+1)$-dimensional total space. Let $d\lambda_z$ be
the Euclidean volume form in $z$, and define the smooth positive
function $\rho$ by $\eta_t^n=\rho(z,t)\,d\lambda_z$.
On a fixed vertical ball $\Omega$ slightly larger than the region
under consideration, we define, for each integer $k\ge1$, the
diagonal weighted Bergman kernel by
\[
 K_{kv}(z,t):=
 \sup\left\{|g(z)|^2:
   g\in\mathcal O(\Omega),\quad
   \int_\Omega |g(\zeta)|^2e^{-kv(\zeta,t)}\,d\lambda_\zeta\le1
       \right\},
\]
where $\mathcal O(\Omega)$ denotes the space of holomorphic functions
on $\Omega$. With our conventions for $\ddc$ and $d\lambda_z$ we put
$c_n:=n!$; with this normalization
$c_nk^{-n}K_{kv}e^{-kv}\,d\lambda_z$ approximates $T_t^n$.
Using the unregularized density $f$, we put, outside $\Supp D$,
\[
 F_k:=\log\frac{c_nk^{-n}K_{kv}e^{-kv}}{e^Uf\rho},
 \qquad F_k^+:=\max\{F_k,0\}.
\]
Here and below, $r^+=\max\{r,0\}$ and $\log^\pm x=\max\{\pm\log x,0\}$.
Then $F_k^+S$ is locally integrable, the current $[E]\wedge S$ is well
defined, and
\begin{equation}\label{eq:signed-entropy-current}
 \ddc(F_k^+S)\ge-[E]\wedge S.
\end{equation}
\end{lemma}

\begin{proof}
Let $g_i$ be a local holomorphic equation of $E_i$, and put
$\ell_E:=\sum_i d_i\log|g_i|^2$. Then $\ddc\ell_E=[E]$ by the
Poincar\'e--Lelong formula. By \eqref{eq:analytic-model-identity},
where the canonical curvature is computed from $\rho$, the function
\[
 \Xi_E=a_\delta+\ell_E-\log(f\rho)\quad\text{satisfies}\quad
 \ddc\Xi_E=b+\delta\eta+[C]\ge0;
\]
hence $\Xi_E$ extends as a psh function across the divisors. By the
positivity of Bergman kernels, the first entry in the following
maximum is psh; the second entry is psh, since $v$ and $\ell_E$ are
psh. Consequently
\[
 W_k=\max\{\log(c_nk^{-n}K_{kv})+\Xi_E,
                          (k+1)v+\ell_E\}
\]
is psh. Outside the divisors, and hence as locally integrable
functions, we have
\[
 F_k^+=W_k-(k+1)v-\ell_E.
\]

Before applying $\ddc$ to this identity, we must justify the product
$W_kS$. The psh function $W_k$ is locally bounded from above. Since
$W_k\ge(k+1)v+\ell_E$ and $v$ is locally bounded, the negative part
of $W_k$ is controlled by integrable logarithms. For fixed
regularization parameters the coefficients of $S$ are bounded, so
that $W_kS$ has locally integrable coefficients. For positive
integers $N$, the dominated convergence theorem gives
$\max(W_k,-N)S\to W_kS$ in $L^1_{\mathrm{loc}}$, coefficientwise in
local coordinates. Each truncation is locally bounded and psh.
Hence, by the Bedford--Taylor theory,
$\ddc(\max(W_k,-N)S)\ge0$. Passing to the limit in the sense of
distributions, we obtain $\ddc(W_kS)\ge0$.

The product $[E]\wedge S$ is defined by restricting the locally
bounded potential $v$ to the smooth components of $E$. Smoothing $v$
locally and using the continuity theorem of Bedford--Taylor for
bounded functions, we get
\[
 \ddc(\ell_ES)=[E]\wedge S.
\]
Similarly, $\ddc(vS)=T^{n+1}=0$. Applying $\ddc$ to the identity for
$F_k^+$ now gives \eqref{eq:signed-entropy-current}. Only the zero
divisor $E$ contributes a negative term; the pole divisor $C$ enters
$\ddc\Xi_E$ with a positive sign.
\end{proof}

\subsection{A fiberwise estimate from Green's formula}

Next we integrate \eqref{eq:signed-entropy-current} over the fibers
and apply Green's formula on the base. The continuity of the
resulting fiber integrals makes the estimate valid at every
parameter. We follow \cite[Lemma~6.4]{LPWZ}, but we retain the
positive part of the logarithm in the boundary term, since in the
presence of poles $f_{a_1,a_2}/f$ can be less than one.

\begin{lemma}
\label{lem:signed-green-details}
Fix $\tau\in \mathbb D$, let $\nu_\tau$ be the harmonic measure of
$\mathbb D$ with respect to $\tau$, that is, the probability measure on
$\partial\mathbb D$ with $H(\tau)=\int_{\partial\mathbb D}H\,d\nu_\tau$
for every function $H$ continuous on $\overline{\mathbb D}$ and
harmonic in $\mathbb D$, and let $g_\tau\ge0$ be the Green function
of $\mathbb D$ with pole at $\tau$, extended by zero outside
$\mathbb D$ and normalized by $\ddc g_\tau=\nu_\tau-\delta_\tau$,
where $\delta_\tau$ is the Dirac mass at $\tau$. A quotient of
measures below means the quotient of their local densities, and its
positive logarithm is set to zero where the numerator vanishes. Then
\begin{align}
 \int_{Z_\tau}
  \log^+\frac{T_\tau^n}{e^{U_\tau}f_\tau\eta_\tau^n}\,T_\tau^n
 &\le\int_{\partial \mathbb D}\int_{Z_t}
      e^{u_{\delta,a_1,a_2,t}}f_{a_1,a_2,t}
       \log^+\frac{f_{a_1,a_2,t}}{f_t}
        \eta_t^n\,d\nu_\tau(t)\notag\\
 &\quad+\int_{q^{-1}(\mathbb D)}q^*g_\tau\,T^n\wedge[E].
 \label{eq:signed-green-estimate}
\end{align}
\end{lemma}

\begin{proof}
\emph{Step 1: Bergman approximation.}
We choose finitely many relative SNC product charts with nested
vertical balls as in the local theorem on Bergman kernels. Let
$(\chi_j)$ be a smooth nonnegative partition of unity subordinate to
the smaller charts, so that $\sum_j\chi_j=1$ near
$q^{-1}(\overline{\mathbb D})$. In the chart $j$, let $a_{\delta,j}$
be the local potential of $P_\delta$, and put
$v_j:=a_{\delta,j}+U$. We define the Bergman density $b_{j,k}$ and the
two measure densities $r_j,\varrho_j$ by
\[
 b_{j,k}:=c_nk^{-n}K_{kv_j}e^{-kv_j},\qquad
 T_t^n=r_j\,d\lambda_z,
 \qquad e^{U_t}f_t\eta_t^n=\varrho_j\,d\lambda_z.
\]
Thus $\varrho_j=e^{U_t}f_t\rho_j$, where
$\eta_t^n=\rho_j\,d\lambda_z$ in this chart. In what follows, $L^1$
norms in coordinates are taken with respect to $d\lambda_z$ on the
smaller vertical ball. By the asymptotics of Bergman kernels, we
have, for every fixed $t$,
\[
 \|b_{j,k}(\cdot,t)-r_j(\cdot,t)\|_{L^1}\longrightarrow0,
 \qquad 0\le b_{j,k},r_j\le C_{\delta,a_1,a_2}.
\]
The ratio $r_j/\varrho_j$ does not depend on the chart; it therefore
defines a global function $F^+:=\log^+(r_j/\varrho_j)$, which we set
equal to zero where $r_j=0$. Its local Bergman approximation is
$F_{j,k}^+:=\log^+(b_{j,k}/\varrho_j)$. Away from the divisor,
$\varrho_j$ is bounded from below on compact sets; hence the $L^1$
convergence of $b_{j,k}$ implies that of $F_{j,k}^+$ to $F^+$.
Near the divisor, both functions have the majorant
\[
 0\le F_{j,k}^+,F^+
 \le C_{\delta,a_1,a_2}
                +\sum_{i:E_i\text{ meets chart }j}
                     d_i\log^-|z_{\ell(i)}|^2=:M_j(z);
\]
where $z_{\ell(i)}=0$ defines $E_i$ in the SNC chart. Only the zero
components occur, since poles of $f$ decrease the positive logarithm.
The majorant is integrable in the vertical coordinates, uniformly in
$t$. Its integral over a small neighborhood of the divisor is
uniformly small, and therefore the convergence away from the divisor
extends to $L^1$ convergence on the whole smaller vertical ball on
each fiber. Since the densities $r_j$ are bounded, we also obtain
convergence with respect to $T_t^n$.

We shall need the same convergence after integration in $t$. The
function $1+g_\tau+|dg_\tau|$ is integrable with respect to the
Euclidean area on the base, where $|dg_\tau|$ denotes the Euclidean
norm. Hence, by dominated convergence, we obtain convergence on the
total space with each of these weights. The same argument gives
convergence after integration against $\nu_\tau$ on
$\partial \mathbb D$.

\medskip\noindent
\emph{Step 2: Continuity in the parameter.}
Put
\[
 H_k(t)=\sum_j\int_{Z_t}\chi_jF_{j,k}^+T_t^n.
\]
Since Green's formula involves the value $H_k(\tau)$, we need the
continuity of $H_k$, and not only a representative defined almost
everywhere. Fix $k$. If $t,s$ belong to the same base chart and
$\|v_j(\cdot,t)-v_j(\cdot,s)\|_\infty\le r$ for a number $r\ge0$,
then the extremal characterization of the Bergman kernel gives
\[
 e^{-kr}K_{kv_j}(z,s)\le K_{kv_j}(z,t)
                         \le e^{kr}K_{kv_j}(z,s),
\]
so that the kernel depends continuously on the parameter.

To prove the continuity of the fiber integrals, we work in a fixed
product chart and let $t\to s$. Since $v_j$ is continuous in both
variables, $v_j(\cdot,t)\to v_j(\cdot,s)$ locally uniformly. Hence,
by the continuity theorem of Bedford--Taylor, $T_t^n$ converges
weakly to $T_s^n$ in these coordinates.

In order to deal with the singular integrand, we choose a small
neighborhood $N$ of the divisor in the vertical coordinates and
insert a smooth cutoff function which vanishes near the divisor and
is equal to one outside $N$. The resulting integrand is continuous in
$(z,t)$, and hence its fiber integral depends continuously on the
parameter by the weak convergence just proved. The part omitted by
this cutoff is bounded by
$
 C\int_N M_j(z)\,d\lambda_z,
$
since $F_{j,k}^+\le M_j$ and the densities of $T_t^n$ are uniformly
bounded for fixed regularization parameters. Since $M_j$ is
integrable, this bound tends to zero as $N$ shrinks to the divisor,
uniformly for $t$ near $s$. Thus the original fiber integral is a
uniform limit of continuous functions, and hence it is continuous.
Summing over the finitely many charts, we see that $H_k$ is
continuous on $\overline{\mathbb D}$.

\medskip\noindent
\emph{Step 3: Globalization.}
Let $\varphi\ge0$ be a smooth test function with compact support in
$\mathbb D$, and put $\widehat\varphi:=q^*\varphi$. We apply
\eqref{eq:signed-entropy-current} on the chart $j$ to
$\widehat\varphi\chi_j$, integrate by parts, and sum over the charts.
Then we obtain
\[
 \int_{\mathbb D} H_k\,\ddc\varphi
 \ge-\int_{q^{-1}(\mathbb D)}\widehat\varphi\,S\wedge[E]
       -\mathcal E_k(\varphi),
\]
where
\[
 \mathcal E_k(\varphi):=\sum_j\int_{q^{-1}(\mathbb D)} F_{j,k}^+
 \bigl(\widehat\varphi\ddc\chi_j
       +d\widehat\varphi\wedge d^c\chi_j
       +d\chi_j\wedge d^c\widehat\varphi\bigr)\wedge S.
\]
The error term consists of those terms in
$\ddc(\widehat\varphi\chi_j)$ in which $\chi_j$ is differentiated; no
derivatives of the singular functions $F_{j,k}^+$ occur. If we
replace each local approximation $F_{j,k}^+$ by its common limit
$F^+$, these terms cancel, since $\sum_j\chi_j=1$. By the weighted
$L^1$ convergence in Step~1, $\mathcal E_k(g_\tau)$ is well defined
and $\mathcal E_k(g_\tau)\to0$ as $k\to\infty$. In the next step we
justify the use of the nonsmooth function $g_\tau$ in the inequality
itself.

\medskip\noindent
\emph{Step 4: Green's formula at $\tau$.}
First we keep $k$ fixed. We wish to use $g_\tau$ as the test function
in Step~3; however, $g_\tau$ is singular at $\tau$, and its support
reaches the boundary of $\mathbb D$. We therefore begin with the
truncated functions
\[
 g_{\tau,\kappa,N}:=\min\{(g_\tau-\kappa)^+,N\},
 \qquad \kappa>0,\quad N>0.
\]
The function $(g_\tau-\kappa)^+$ vanishes near the boundary, and the
truncation at height $N$ removes the singularity at $\tau$. We then
smooth these functions inside $\mathbb D$. Letting
$\kappa\downarrow0$, $N\uparrow\infty$, and the smoothing radius tend
to zero, we can choose nonnegative functions
$\varphi_\ell\in C_c^\infty(\mathbb D)$ such that
\[
 \varphi_\ell\longrightarrow g_\tau\quad\text{in }W^{1,1}(\mathbb D),
 \qquad
 \ddc\varphi_\ell\rightharpoonup\nu_\tau-\delta_\tau
 \quad\text{as measures on }\overline{\mathbb D}.
\]
The measures $\ddc\varphi_\ell$ have uniformly bounded total
variation.

We apply Step~3 with $\varphi=\varphi_\ell$ and let $\ell\to\infty$.
Since $H_k$ is continuous on the closed disk, the convergence of the
measures gives
\[
 \int_{\mathbb D} H_k\,\ddc\varphi_\ell
 \longrightarrow
 \int_{\partial \mathbb D}H_k\,d\nu_\tau-H_k(\tau).
\]
Here we need the continuity proved in Step~2. The divisor integral
also converges, since, in relative SNC coordinates, $S\wedge[E]$ has
bounded local densities on each smooth component of $E$, and
$\varphi_\ell\to g_\tau$ in $L^1(\mathbb D)$. Finally,
$\mathcal E_k$ involves only the test function and its first
derivatives. By the bound $F_{j,k}^+\le M_j$, the uniform
integrability of $M_j$ in the vertical variables and the boundedness
of the coefficients of $S$, we have
\[
 |\mathcal E_k(\varphi_\ell)-\mathcal E_k(g_\tau)|
 \le C\|\varphi_\ell-g_\tau\|_{W^{1,1}(\mathbb D)}\longrightarrow0.
\]
Thus, rearranging the inequality in Step~3, we obtain
\[
 H_k(\tau)\le\int_{\partial \mathbb D}H_k\,d\nu_\tau
       +\int_{q^{-1}(\mathbb D)}q^*g_\tau\,S\wedge[E]
       +\mathcal E_k(g_\tau).
\]

Now we let $k\to\infty$. By Step~1 the fiber integral at $\tau$ and
the boundary integral converge, and by Step~3 we have
$\mathcal E_k(g_\tau)\to0$. Consequently,
\[
 \int_{Z_\tau}F^+T_\tau^n
 \le\int_{\partial \mathbb D}\int_{Z_t}F^+T_t^n\,d\nu_\tau
       +\int_{q^{-1}(\mathbb D)}q^*g_\tau\,S\wedge[E].
\]
Finally, on $\partial \mathbb D$ we have $U_t=u_{\delta,a_1,a_2,t}$ and
\[
 T_t^n=e^{u_{\delta,a_1,a_2,t}}f_{a_1,a_2,t}\eta_t^n,
 \qquad
 F^+|_{Z_t}=\log^+\frac{f_{a_1,a_2,t}}{f_t}.
\]
Substituting these expressions into the boundary integral gives
\eqref{eq:signed-green-estimate}. We retain $\log^+$, since the
ratio $f_{a_1,a_2,t}/f_t$ can be less than one.
\end{proof}

\subsection{Vanishing of the boundary and divisor errors}

The two error terms in \eqref{eq:signed-green-estimate} vanish at
different stages. The boundary integral is at most $Ca_1$, uniformly
in the other parameters; hence it disappears when the regularization
of the zeros is removed. The divisor integral is controlled by the
intersection numbers $\int_{E_{i,t}}P_{\delta,t}^{n-1}$, which tend
to zero as $\delta\downarrow0$ by the orthogonality in
\eqref{eq:analytic-model-identity}. For the second estimate we use
\cite[Lemma~6.6]{LPWZ}; for the first we must check that the pole
factors do not affect the bound.

\begin{lemma}
\label{lem:signed-error-details}
The boundary term in \eqref{eq:signed-green-estimate} is at
most $Ca_1$, uniformly in $\delta$ and $a_2$. If
$
 m_{i,\delta}=\int_{E_{i,t}}P_{\delta,t}^{n-1},
$
then
\[
 \int_{q^{-1}(\mathbb D)}q^*g_\tau\,T^n\wedge[E]
       \le C_{\tau,\mathbb D}\sum_i d_i m_{i,\delta},
 \qquad m_{i,\delta}\longrightarrow0\quad(\delta\downarrow0).
\]
The number $m_{i,\delta}$ is independent of $t\in \mathbb D$, since it
is the integral of a closed form over the fibers of $E_i\to \mathbb D$.
The constant $C_{\tau,\mathbb D}$ is independent of $a_1,a_2,\delta$
and of the $C^{1,1}$ norm of $U$.
\end{lemma}

\begin{proof}
Every pole factor of $f_{a_1,a_2}/f$ is at most one; hence
\[
 \log^+(f_{a_1,a_2}/f)
       \le\sum_i d_i\log(1+a_1/|s_{E_i}|^2).
\]
By \eqref{eq:model-regularization-bound} the exponential factor in
the boundary term is uniformly bounded. We expand the above sum and
estimate one term at a time in a relative SNC chart. Each term
consists of one logarithmic factor multiplied by the full regularized
density. Integrating in the corresponding zero coordinate, we are
reduced, up to uniformly bounded smooth factors and fixed constants,
to the following estimate, where $x$ is the squared radial
coordinate, $R>0$ is a fixed squared radius, and $d=d_i>0$:
\[
 \int_0^R(x+a_1)^d\log(1+a_1/x)\,dx\le C_d a_1.
\]
On $\{x<a_1\}$ the substitution $x=a_1y$ gives
$O(a_1^{d+1})$. On $\{x>a_1\}$ we use
$\log(1+a_1/x)\le a_1/x$ and integrate $2^d a_1x^{d-1}$.
For this summand the other zero coordinates contribute only bounded
power factors; their logarithms occur in the other summands, which
are estimated in the same way. Each pole factor is bounded by
$y^{-c_j}$ in its squared radial coordinate, and
$\int_0^R y^{-c_j}\,dy$ is finite, since $c_j<1$. Summing over the
logarithmic terms and over a finite covering by charts, we obtain the
bound $Ca_1$.

The divisor estimate and the convergence $m_{i,\delta}\to0$ follow
from \cite[Lemma~6.6]{LPWZ}, with the effective divisor there
replaced by $E$. The hypotheses of that lemma are the relative SNC
condition, the orthogonality in \eqref{eq:analytic-model-identity}
and the uniform bound \eqref{eq:signed-perron-uniform-bound}; its
constant involves neither the density nor the $C^{1,1}$ norm of the
envelope.
\end{proof}

\subsection{The limiting comparison and joint positivity}

We now let $a_1\downarrow0$, then $a_2\downarrow0$, and finally
$\delta\downarrow0$. Following \cite[Proposition~7.1]{LPWZ}, we prove
in the next lemma that $u$ is jointly $P$-psh, by means of the
vanishing of the error terms and the comparison with the limiting
Perron envelope.

\begin{lemma}
\label{lem:signed-limit-details}
Define the iterated limits of the Perron envelopes by
\[
 U_\delta:=\lim_{a_2\downarrow0}\lim_{a_1\downarrow0}
                 U_{\delta,a_1,a_2},\qquad
 U_0:=\lim_{\delta\downarrow0}U_\delta,
\]
and write $U_{0,\tau}:=U_0|_{Z_\tau}$ for $\tau\in \mathbb D$.
These limits exist, and on every fiber $Z_\tau$ we have
\[
 (P_\tau+\ddc U_{0,\tau})^n
             \le e^{U_{0,\tau}}f_\tau\eta_\tau^n,
 \qquad u_\tau\le U_{0,\tau}.
\]
These inequalities imply that the function $u$ in
Lemma~\ref{lem:signed-regularization-details} is jointly $P$-psh.
\end{lemma}

\begin{proof}
\emph{Step 1: Convergence of the envelopes.}
For a fixed background form, the Perron envelope increases with the
boundary data. Moreover, if two boundary functions differ by at most
a constant $c$, then, shifting the competitors by $c$, we see that
their envelopes also differ by at most $c$. Thus the envelope is
$1$-Lipschitz with respect to the supremum norm of the boundary data.
Therefore the two successive uniform limits of $u_{\delta,a_1,a_2}$
give
uniform limits of $U_{\delta,a_1,a_2}$, and we obtain an envelope
$U_\delta$ with boundary values $u_\delta$.

Now we let the background parameter vary. For $\delta'>\delta$,
every $P_\delta$-psh competitor is also $P_{\delta'}$-psh, and
$u_\delta\le u_{\delta'}$ on the boundary. Hence
$U_\delta\le U_{\delta'}$. By the bound
\eqref{eq:signed-perron-uniform-bound}, these envelopes decrease to a
bounded function $U_0$ as $\delta\downarrow0$. This function is
$P$-psh. In fact, locally $U_\delta+A+\delta R$ is a decreasing
sequence of psh functions with limit $U_0+A$, where $A,R$ are local
potentials of $P,\eta$ and $R\ge0$.

\medskip\noindent
\emph{Step 2: The limiting Monge--Amp\`ere inequality.}
For nonnegative numbers $r,s$ we have $(r-s)^+\le r\log^+(r/s)$,
where the right-hand side is defined to be zero if $r=0$ and $+\infty$
if $r>0=s$. We apply this inequality, for fixed regularization
parameters, to the densities of $T_\tau^n$ and
$e^{U_\tau}f_\tau\eta_\tau^n$. Then the entropy estimate
\eqref{eq:signed-green-estimate}, together with
Lemma~\ref{lem:signed-error-details}, bounds the total mass of the
positive part of their difference by
\[
 Ca_1+C_{\tau,\mathbb D}\sum_i d_i m_{i,\delta}.
\]
In order to pass to the limit, we express this bound by means of test
functions: the integral of the difference against any continuous
function with values in $[0,1]$ is bounded by the same quantity. We
let first $a_1\downarrow0$ and then $a_2\downarrow0$. By the uniform
convergence of the potentials, the Monge--Amp\`ere measures on each
fiber converge in the sense of Bedford--Taylor. The other measures
converge by dominated convergence, since $f_\tau$ is fixed and
integrable and the exponential factors are uniformly bounded.

Finally we let $\delta\downarrow0$. As in Step~1, adding local
background potentials, we obtain a decreasing sequence of bounded psh
functions, and the convergence theorem of Bedford--Taylor applies
again. The dominated convergence theorem still applies to the terms
with densities. Since $m_{i,\delta}\to0$ by orthogonality, no error
term remains, and we obtain the limiting inequality
\[
 (P_\tau+\ddc U_{0,\tau})^n
                \le e^{U_{0,\tau}}f_\tau\eta_\tau^n.
\]
Thus $U_{0,\tau}$ is a supersolution of
\eqref{eq:analytic-model-equation}, whose solution is $u_\tau$.
Hence $u_\tau\le U_{0,\tau}$ by the comparison principle in the
semipositive big class $[P_\tau]$.

\medskip\noindent
\emph{Step 3: Joint positivity.}
The function $u$ is upper semicontinuous, since it is the decreasing
limit of the continuous functions $u_\delta$. It remains to prove the
submean inequality in the directions involving the base variable.
The above construction works for any relatively compact base disk.
We work in a product chart and choose such a disk $\mathbb D$,
together with a holomorphic graph $\sigma:\overline{\mathbb D}\to Z$,
defined on a neighborhood of $\overline{\mathbb D}$, whose image lies
in the chart and which satisfies $q\circ\sigma(t)=t$. Let $A,R$ be
local potentials of $P,\eta$. The psh function $U_\delta+A+\delta R$
satisfies the submean inequality along $\sigma$. Since
$U_\delta=u_\delta$ over $\partial \mathbb D$, we have
\[
 (U_\delta+A+\delta R)(\sigma(\tau))
   \le\int_{\partial \mathbb D}
             (u_\delta+A+\delta R)(\sigma(t))\,d\nu_\tau(t).
\]
Letting $\delta\downarrow0$ by bounded convergence and using
$u_\tau\le U_{0,\tau}$, we obtain
\[
 (u+A)(\sigma(\tau))
  \le(U_0+A)(\sigma(\tau))
  \le\int_{\partial \mathbb D}(u+A)(\sigma(t))\,d\nu_\tau(t).
\]
Every nonvertical affine complex line in the chart is locally a graph
over the base. Hence the above inequality proves that $u+A$ is
subharmonic along these lines. On vertical lines the subharmonicity
follows from $u_t\in\PSH(Z_t,P_t)$. By the characterization of psh
functions by their restrictions to complex lines, $u+A$ is psh. The
comparison principle has been used only after passing to the
limits; no inequality between $u_\delta$ and $U_\delta$ for
$\delta>0$ is needed.
\end{proof}

\begin{proof}[Proof of Proposition~\ref{prop:signed-model}]
By Lemma~\ref{lem:signed-regularization-details}, there is a unique
bounded solution on each fiber. By
Lemmas~\ref{lem:signed-local-current-details}--\ref{lem:signed-limit-details},
these solutions define a jointly $P$-psh function over every
relatively compact disk in the base. By the uniqueness on the fibers,
the functions on overlapping disks coincide. This proves the
proposition on $Z$.
\end{proof}

\section{Descent and identification of the fiberwise metrics}
\label{sec:log-descent}

In Section~\ref{sec:signed-analysis} we have constructed a metric on
a resolved model outside a discrete set of parameters. In this
section we descend this metric to the original smooth family. By
means of the corrections for the Jacobian and for the boundary, we
identify the descended potential and its equation. We then extend the
potential across the omitted fibers and prove that the equation holds
there with full mass. In this way we obtain the perturbed metric
\eqref{eq:joint-log-ke-final} needed for the extension.

\subsection{Descent to the original family}

We return to the original family $p:\X\to\Delta$ and to the reference
data \eqref{common:reference-data}. Thus $F_B=e^{-\phi_B}$ with
$\phi_B=\sum_i b_i\log|\sigma_i|_{h_i}^2$, and
\[
 \theta=\theta_\omega+\Theta_B+\beta,
\]
where $\theta_\omega=\Theta(h_K)$ is the curvature of the metric on
$K_{\X/\Delta}$ induced by the vertical volume forms $\omega_t^n$,
$\Theta_B=\sum_i b_i\Theta(h_i)$, and $\beta=\Theta(h_M)$. Fix
$\varepsilon>0$, put $\alpha=[\theta+\varepsilon\omega]$, and take
the model of Proposition~\ref{prop:local-log-model}:
\[
 \X\xleftarrow{\mu}Z\xrightarrow{f}Y,
 \qquad q=p\circ\mu,\qquad P=f^*\gamma.
\]
Here $\gamma$ is the K\"ahler form on $Y$. Let $E_\alpha$ be the
effective $f$-exceptional divisor in the decomposition of the model,
and let $E_\mu$ be the relative discrepancy divisor of $\mu$. We
define the signed divisor $D$ by
\[
 \mu^*\alpha=[P]+[E_\alpha],
 \qquad E_\mu=K_{Z/\Delta}-\mu^*K_{\X/\Delta},
 \qquad D=E_\alpha+E_\mu-\mu^*B.
\]
We use the corresponding excluded set $\Sigma$ and the set of good
parameters $\Delta^\circ=\Delta\setminus\Sigma$ defined in
Section~\ref{sec:relative-log-geometry}. For every $t\in\Delta^\circ$,
$\mu_t:Z_t\to X_t$ and $f_t:Z_t\to Y_t$ are modifications.
The following construction takes place on
$Z^\circ:=q^{-1}(\Delta^\circ)$; its potentials and the vertical
Jacobian weight are at first defined only on this open set. Over
$\Delta^\circ$, the hypotheses of Proposition~\ref{prop:signed-model}
are satisfied by the construction of
Section~\ref{sec:relative-log-geometry}, with
\[
 E=E_\alpha+E_\mu,\qquad C=\mu^*B,\qquad
 b=\mu^*\beta+\varepsilon\mu^*\omega.
\]
Thus $E$ and $C$ are the zero and pole divisors of the density on the
model; their common components may be cancelled as in that
proposition. By the klt condition the coefficients of
$E_\mu-\mu^*B$ are greater than $-1$. Since $E_\alpha\ge0$, the same
holds for $D$.

We choose the data $\eta$, $\Theta_D$ and $h$ as in
\eqref{eq:analytic-model-identity}. Namely, $\eta$ is a K\"ahler form
on $Z$, $\theta_\eta$ is the curvature of the metric on the relative
canonical bundle induced by $\eta_t^n$, $h$ is smooth, and
\[
 \theta_\eta+b=P+\Theta_D+\ddc h,
\]
where $\Theta_D$ and $|s_D|^2$ are defined as in
Proposition~\ref{prop:signed-model}.

Let $u$ be the potential given by
Proposition~\ref{prop:signed-model}. On the complement of the
relevant divisors we define the Jacobian weight $J$ on each fiber by
$J|_{Z_t}=\log(\mu_t^*\omega_t^n/\eta_t^n)$, and put
\[
 \psi_Z=u+\log|s_D|^2-h,
 \qquad \Phi_{\rm pull}=\psi_Z-J+\mu^*\phi_B.
\]
The logarithmic expressions are understood with their natural
extensions as locally integrable weights. From the Jacobian identity
\[
 \ddc J=\mu^*\theta_\omega-\theta_\eta+[E_\mu]
\]
we obtain the identities for the curvature and the density
\begin{equation}\label{eq:log-corrected-descent}
 \mu^*(\theta+\varepsilon\omega)+\ddc\Phi_{\rm pull}
       =P+\ddc u+[E_\alpha]\ge0,
 \qquad
 e^{\Phi_{{\rm pull},t}}\mu_t^*(F_{B,t}\omega_t^n)
       =e^{\psi_{Z,t}}\eta_t^n.
\end{equation}
The only logarithmic divisor term which remains in $\Phi_{\rm pull}$
comes from the effective divisor $E_\alpha$, so that
$\Phi_{\rm pull}$ is locally bounded from above. Where $\mu$ is an isomorphism,
$\Phi_{\rm pull}$ defines a $(\theta+\varepsilon\omega)$-psh function
on the corresponding open subset of $\X$. Since $\mu$ is proper, this
function is locally bounded from above near the image of the
exceptional locus, and hence, by the removable singularity theorem, it
extends across that image; we denote the extension by
$\phi_\varepsilon$. By the model equation and
\eqref{eq:log-corrected-descent}, the Monge--Amp\`ere measures
coincide outside the divisors. Since nonpluripolar products put no
mass on divisors, pushing forward by the modifications of the fibers,
we obtain
\[
 \left\langle
    (\theta_t+\varepsilon\omega_t+\ddc\phi_{\varepsilon,t})^n
 \right\rangle
       =e^{\phi_{\varepsilon,t}}F_{B,t}\omega_t^n.
\]
The total mass is $\int_{Z_t}P_t^n$, which is equal to
$\vol([\theta_t+\varepsilon\omega_t])$ by
Proposition~\ref{prop:local-log-model}; thus $\phi_{\varepsilon,t}$ is
a full-mass solution. By the uniqueness for the exponential equation
in a big class, the functions obtained from different local models
coincide.

Moreover, $\phi_{\varepsilon,t}$ has minimal singularities on every
good fiber. In fact, $\phi_{\varepsilon,t}$ is bounded from above and
$F_{B,t}\in L^q$ for some $q>1$, so that the measure
$e^{\phi_{\varepsilon,t}}F_{B,t}\omega_t^n$ has an $L^q$ density. For
this fixed measure there is, by \cite[Theorems~3.1 and~4.1]{BEGZ}, a
full-mass solution with minimal singularities. By the uniqueness for
the equation with fixed measure, its potential differs from
$\phi_{\varepsilon,t}$ by a constant. Hence $\phi_{\varepsilon,t}$
also has minimal singularities. We note that the $L^q$ property of
the density is used here; full mass alone would not suffice.

\subsection{Stability at a missing parameter}

In order to identify the equation on a fiber omitted in the
construction of the model, we adapt \cite[Lemma~7.2]{LPWZ}, where the
reference measures on the fibers are smooth. Away from $B$ we give
the local subsolution argument in detail; we use only nonzero
parameters in the envelopes, and we remove the truncations by
plurifine locality. The boundary density $F_B$ is treated by its
integrability when we compare masses on whole fibers.

\begin{lemma}
\label{lem:weighted-missing-fiber}
Let $p:\X\to\Delta$ be a smooth proper family with $n$-dimensional
fibers $X_t$, let $a$ be a smooth closed real $(1,1)$-form on $\X$
with big fiber classes $[a_t]$, and let $B$ be an effective relative
SNC boundary with coefficients less than one. We use the reference
K\"ahler form $\omega$ and the boundary density $F_B$ defined in
\eqref{common:reference-data}. Suppose that $\Phi\in\PSH(\X,a)$ and
that its restrictions for $t\ne0$ satisfy
\[
 \langle(a_t+\ddc\Phi_t)^n\rangle
      =e^{\Phi_t}F_{B,t}\omega_t^n
\]
with full mass. If
$\vol([a_t])\to\vol([a_0])=V_0>0$, then the same equation holds,
with full mass, at $0$.
\end{lemma}

\begin{proof}
\emph{Step 1: A lower bound for the exponential mass.}
Shrinking the disk, we may assume that $\Phi$ is bounded from above.
In order to compare integrals over different fibers, we cover $X_0$
by finitely many relative SNC product charts and choose a smooth
partition of unity with compact supports in these charts; after
shrinking again, it is a partition of unity on every fiber under
consideration. In each chart we have
\[
 F_{B,t}\omega_t^n
       =g(z,t)\prod_{i=1}^k|z_i|^{-2b_i}\,d\lambda_z,
\]
where $z=(z_1,\ldots,z_n)$ are fiber coordinates, $d\lambda_z$ is the
Euclidean volume form, the first $k$ coordinate hyperplanes are the
boundary components in this chart, and $g$ is smooth and positive,
with uniform bounds on the chosen supports. Since $b_i<1$, the
function $C\prod_i|z_i|^{-2b_i}$ is an integrable upper bound for
the exponential density, independent of $t$. The joint upper
semicontinuity of $\Phi$ gives the pointwise $\limsup$ inequality
needed for the reverse Fatou lemma away from the coordinate divisors.
Applying that lemma in each chart and summing the localized
integrals, we obtain
\[
 V_0=\lim_{t\to0}\int_{X_t}e^{\Phi_t}F_{B,t}\omega_t^n
       \le\int_{X_0}e^{\Phi_0}F_{B,0}\omega_0^n,
\]
where the equality holds since the solutions for $t\ne0$ have full
mass. The right-hand side is finite, by the same majorant, and
positive; in particular, $\Phi_0\not\equiv-\infty$.

\medskip\noindent
\emph{Step 2: The central subsolution inequality.}
Next we estimate the central Monge--Amp\`ere measure from below by the
exponential density. We work in a relatively compact product chart
disjoint from $B$, and write $a=\ddc A$ with $A$ smooth and real.
Then $v=\Phi+A$ is jointly psh, and for $t\ne0$ the equation takes the
form
\[
 \langle(\ddc_zv_t)^n\rangle
       =e^{v_t}G(z,t)\,d\lambda_z,
\]
where $\ddc_z$ denotes differentiation in the fiber variables only and
$G=e^{-A}F_B\rho$, with $\omega_t^n=\rho\,d\lambda_z$ in this
chart; thus $G$ is smooth and positive away from $B$.
On a smaller product chart, $v$ is uniformly bounded from above and
$G$ is bounded. We choose $C_0$ so large that, for every $C\ge C_0$,
the smooth strictly psh function $\rho_C(z)=|z|^2-C$ satisfies
\[
 (\ddc\rho_C)^n\ge e^{\rho_C}G(z,t)\,d\lambda_z
\]
on this chart. Fix such a $C$, and put
$w_C(z,t)=\max\{v(z,t),\rho_C(z)\}$. This function is jointly psh and
uniformly bounded on the chart.

We first show that $w_{C,t}$ is a subsolution for every $t\ne0$.
Fix such a parameter. There are at most countably many level sets of
$v_t-\rho_C$ with positive Lebesgue measure; hence we can choose
$\epsilon_j\downarrow0$ such that
$\{v_t=\rho_C-\epsilon_j\}$ has Lebesgue measure zero. Put
$w_{t,j}=\max\{v_t,\rho_C-\epsilon_j\}$. On
$\{v_t>\rho_C-\epsilon_j\}$, the Monge--Amp\`ere measure of
$w_{t,j}$ coincides with that of $v_t$ by the plurifine locality of
nonpluripolar products \cite{BEGZ}. On $\{v_t<\rho_C-\epsilon_j\}$,
the subsolution inequality follows from
\[
 (\ddc\rho_C)^n
 \ge e^{\rho_C}G(z,t)\,d\lambda_z
 \ge e^{\rho_C-\epsilon_j}G(z,t)\,d\lambda_z.
\]
On the set where equality holds, the right-hand side has zero mass,
while the left-hand side is a positive measure. Consequently
$ (\ddc w_{t,j})^n\ge e^{w_{t,j}}G(z,t)\,d\lambda_z $.
Since $0\le w_{C,t}-w_{t,j}\le\epsilon_j$, the bounded convergence
theorem of Bedford--Taylor \cite{BT82} and dominated convergence
give
\[
 (\ddc w_{C,t})^n\ge e^{w_{C,t}}G(z,t)\,d\lambda_z
 \qquad(t\ne0).
\]

By the same reasoning, the maximum of finitely many bounded psh
subsolutions for one fixed density is a subsolution. In fact, for
two such functions $h_1,h_2$, we shift $h_2$ down by positive
constants tending to zero, so that the equality sets have zero
Lebesgue measure. The shifted function is still a subsolution, since
its exponential decreases. By locality we obtain the inequality for
the shifted maxima, and their uniform convergence to
$\max\{h_1,h_2\}$ gives the assertion. By induction the assertion
holds for any finite number of functions.

For sufficiently small $r>0$ we define
\[
 W_{r,C}(z)=
 \left(\sup_{0<|t|\le r}w_C(z,t)\right)^*,
 \qquad G_r(z)=\min_{|t|\le r}G(z,t),
\]
where the regularization is taken in the variables $z$. The function
$G_r$ is continuous and positive. Each competitor in the punctured
supremum is a bounded subsolution for the density $G_r$. By
Choquet's lemma there are parameters $t_j$ with $0<|t_j|\le r$
such that $W_{r,C}=(\sup_j w_{C,t_j})^*$. Hence the finite maxima
$H_j=\max\{w_{C,t_1},\ldots,w_{C,t_j}\}$ satisfy
\[
 (\ddc H_j)^n\ge e^{H_j}G_r\,d\lambda_z.
\]
They increase to $W_{r,C}$ outside a pluripolar set and have uniform
local upper and lower bounds. By the convergence theorem of
Bedford--Taylor for bounded increasing sequences and by dominated
convergence, using that pluripolar sets have Lebesgue measure zero,
we obtain
\[
 (\ddc W_{r,C})^n\ge e^{W_{r,C}}G_r\,d\lambda_z.
\]

As $r\downarrow0$, these envelopes decrease to $w_{C,0}$. In fact,
for the lower bound, the submean inequality for the bounded
subharmonic function $t\mapsto w_C(z,t)$ gives
$w_C(z,0)\le\sup_{0<|t|\le r}w_C(z,t)\le W_{r,C}(z)$.
For the reverse inequality in the limit, the joint upper
semicontinuity gives, for each $z_0$ and $\epsilon>0$,
$w_C(z,t)\le w_C(z_0,0)+\epsilon$ if $z$ is sufficiently close
to $z_0$ and $|t|$ is sufficiently small; taking the supremum and its
upper regularization, we obtain the same upper bound for
$W_{r,C}(z_0)$. Moreover, $G_r\to G(\cdot,0)$ locally uniformly.
Hence, by the convergence theorem of Bedford--Taylor for bounded
decreasing sequences,
\[
 (\ddc\max\{v_0,\rho_C\})^n
 \ge e^{\max\{v_0,\rho_C\}}G(z,0)\,d\lambda_z.
\]

To remove the truncation, we restrict this inequality to the
plurifine open set $A_C=\{v_0>\rho_C\}$. There, by locality, the
left-hand side is equal to $\langle(\ddc v_0)^n\rangle$ and the
right-hand side to $e^{v_0}G(z,0)\,d\lambda_z$. As
$C\uparrow\infty$, the sets $A_C$ increase to
$\{v_0>-\infty\}$. By the monotone convergence of these restrictions
we obtain the desired inequality on this set, and both measures put
zero mass on $\{v_0=-\infty\}$. Thus
$\langle(\ddc v_0)^n\rangle\ge e^{v_0}G(z,0)\,d\lambda_z$ on
the whole chart. Returning to invariant notation, we obtain
\[
 \langle(a_0+\ddc\Phi_0)^n\rangle
       \ge e^{\Phi_0}F_{B,0}\omega_0^n
\]
on $X_0\setminus B_0$. Neither measure charges $B_0$, since the
left-hand side is a nonpluripolar product and the right-hand side has
an integrable klt density. Hence the inequality holds on the whole of
$X_0$.

\medskip\noindent
\emph{Step 3: Equality.}
Since the total mass of a nonpluripolar Monge--Amp\`ere measure in the
big class $[a_0]$ is at most $\vol([a_0])$, we obtain from Steps~1
and~2
\[
 V_0\le\int_{X_0}e^{\Phi_0}F_{B,0}\omega_0^n
 \le\int_{X_0}\langle(a_0+\ddc\Phi_0)^n\rangle
 \le\vol([a_0])=V_0.
\]
Therefore all these inequalities are equalities. The difference of the
two measures is nonnegative and has total mass zero; hence the two
measures coincide, and $\Phi_0$ has full mass.
\end{proof}

\subsection{Completion of the perturbed metric construction}

We now extend $\phi_\varepsilon$ across the exceptional parameters.
Fix $\varepsilon>0$. On the good fibers we put
$c_t=\sup_{X_t}\phi_{\varepsilon,t}$ and
$v_t=\phi_{\varepsilon,t}-c_t$, so that $\sup_{X_t}v_t=0$.
Let $W$ and $\nu_t$ be as in Section~\ref{sec:common-extension};
recall that $W$ is independent of $t$ and that $F_B\ge1$. By the
uniform Skoda integrability on the fixed smooth K\"ahler family, we
have $\int_{X_t}v_t\,d\nu_t\ge-C$, and Jensen's inequality gives
\[
 \int_{X_t}e^{v_t}F_{B,t}\,d\nu_t
       \ge\exp\!\left(\int_{X_t}v_t\,d\nu_t\right)
       \ge e^{-C}.
\]
By Corollary~\ref{cor:common-domain}, the equation has constant
positive mass $V_\varepsilon=\vol([\theta_t+\varepsilon\omega_t])$.
Since this mass is equal to
$e^{c_t}W\int_{X_t}e^{v_t}F_{B,t}\,d\nu_t$, we obtain
$c_t\le\log(V_\varepsilon/W)+C$. By this upper bound and the
removable singularity theorem, $\phi_\varepsilon$ extends across each
exceptional fiber.

The constancy of the volume gives the remaining hypothesis of
Lemma~\ref{lem:weighted-missing-fiber}; hence the extended potential
satisfies the equation with full mass on every exceptional fiber.
The argument with fixed measure above, based on
\cite[Theorem~4.1]{BEGZ}, also gives minimal singularities there.
Consequently, for every $\varepsilon>0$ and on every fiber of the
fixed disk, we have
\begin{equation}\label{eq:joint-log-ke-final}
 \phi_\varepsilon\in
       \PSH(\X,\theta+\varepsilon\omega),\qquad
 \left\langle
   (\theta_t+\varepsilon\omega_t+\ddc\phi_{\varepsilon,t})^n
 \right\rangle
       =e^{\phi_{\varepsilon,t}}F_{B,t}\omega_t^n,
\end{equation}
with full mass and minimal singularities. In the next section we let
$\varepsilon$ tend to zero by means of
Lemma~\ref{common:metric-criterion}.

\section{Proof of logarithmic extension}
\label{sec:general-proof}

We now combine the geometric and the analytic constructions. By
Section~\ref{sec:relative-log-geometry} the perturbed volumes are
constant on a fixed disk. In Sections~\ref{sec:signed-analysis}
and~\ref{sec:log-descent} we have constructed positively curved
metrics which satisfy the corresponding equations on every fiber. The
criterion of Section~\ref{sec:common-extension} then gives the
extension. Since each step admits a smooth semipositive twist, we
obtain the main theorem in the following form. For $B=0$ and
projective families, the extension assertion follows from
\cite[Theorem~1]{Paun07}, applied to the line bundle $mM$ with its
smooth semipositive metric; here we admit a relative SNC boundary and
a K\"ahler total space.

\begin{theorem}\label{thm:smooth-twist}
In the setting of Theorem~\ref{thm:main}, let $M$ be a rational
holomorphic line bundle with a smooth Hermitian metric $h_M$ of
semipositive curvature on $\X$, and put
$A=K_{\X/\Delta}+B+M$.
Then, for every integer $m\ge1$ such that $mB$ and $mM$ are integral,
every section of $mA_0$ extends
over $p^{-1}(U)$ for a sufficiently small disk $U$ about $0$.
The dimensions $h^0(X_t,mA_t)$ are locally constant, and
$p_*\OO_\X(mA)$ is locally free and commutes with holomorphic
base change. No assumption of nefness or bigness is needed.
\end{theorem}

\begin{proof}[Proof of Theorems~\ref{thm:smooth-twist} and~\ref{thm:main}]
Let \(m\) be admissible, and let \(s\in H^0(X_0,mA_0)\) be a nonzero
section. If \(m=1\), the argument of
Proposition~\ref{common:weighted-extension} applies directly; hence
we may assume that \(m\ge2\).

\medskip\noindent
\emph{Step 1: Reduction to canonical fiber pairs.}
By Lemma~\ref{lem:canonicalization}, after shrinking the disk, there
is a simultaneous crepant modification
\(\tau:\widetilde\X\to\X\) such that every fiber pair
\((\widetilde X_t,\widetilde B_t)\) is canonical. The pullback of the
metric on \(M\) is a smooth metric of semipositive curvature on
\(\tau^*M\), and we have, as rational line bundles,
\[
 K_{\widetilde\X/\Delta}+\widetilde B+\tau^*M=\tau^*A.
\]
The integer \(m\) is again admissible, and the pullback of \(s\) is a
nonzero section in degree \(m\). We choose a K\"ahler form on
\(\widetilde\X\). In Steps~2 and~3 we denote the modified data again
by \((\X,B,M,\omega,A,s)\); in Step~4 we return to the original
family.

\medskip\noindent
\emph{Step 2: The perturbed metrics on a fixed disk.}
Since \(s\ne0\), the class \(c_1(A_0)\) is pseudoeffective. By
Corollary~\ref{cor:common-domain} with \(\beta=\Theta(h_M)\), for
every \(\varepsilon>0\) the class \(c_1(A_t)+\varepsilon[\omega_t]\)
is big on the whole disk, with constant volume \(v_\varepsilon>0\).
For each fixed \(\varepsilon\),
Propositions~\ref{prop:local-log-model} and~\ref{prop:signed-model}
give the perturbed potentials outside the discrete exceptional sets
of the local models. In Section~\ref{sec:log-descent} these
potentials have been descended to the smooth family. By the
uniqueness for the exponential equation they coincide on overlaps,
and they extend across the exceptional parameters by
Lemma~\ref{lem:weighted-missing-fiber}. Thus
\eqref{eq:joint-log-ke-final} holds on a fixed disk for every
\(\varepsilon>0\).

\medskip\noindent
\emph{Step 3: The limit \(\varepsilon\to0\).}
We use the reference metric $h_A$ and its curvature
$\theta=\Theta(h_A)$ in \eqref{common:reference-data} on this
modified family. For any sequence \(\varepsilon_j\downarrow0\), these
potentials and volumes satisfy the hypotheses of
Lemma~\ref{common:metric-criterion}. Hence there is a function
\(\psi\in\PSH(\X,\theta)\) satisfying \eqref{common:section-barrier},
and by Proposition~\ref{common:weighted-extension} the section
\(s\) extends on the canonicalized family. Since all estimates are
made on this fixed smooth family, we need neither uniform control of
the varying bimeromorphic models nor a positive lower bound for
\(v_{\varepsilon_j}\).

\medskip\noindent
\emph{Step 4: Descent.}
We now return to the original family. By the crepant equality in
Step~1, the extended section $\widetilde S$ on $\widetilde\X$ is a
section of $\tau^*\mathcal O_\X(mA)$. Since
$\tau_*\mathcal O_{\widetilde\X}=\mathcal O_\X$, the projection
formula gives
\[
 \tau_*\bigl(\tau^*\mathcal O_\X(mA)\bigr)
       \simeq\mathcal O_\X(mA).
\]
Thus $\widetilde S$ is the pullback of a unique section $S$ on the
original family. Let $\tau_0:\widetilde X_0\to X_0$ be the central
restriction of $\tau$. The sections $S|_{X_0}$ and $s$ have the same
pullback $\widetilde S|_{\widetilde X_0}$ under $\tau_0$. Since
$\tau_0$ is bimeromorphic, the pullback is injective, and hence
$S|_{X_0}=s$.

The same argument applies with any fiber in place of the central
fiber. Fix a parameter $t_0$, and extend a basis of
$H^0(X_{t_0},mA_{t_0})$. The extensions are linearly independent on
nearby fibers. Hence $h^0(X_t,mA_t)\ge h^0(X_{t_0},mA_{t_0})$ near
$t_0$, and the reverse inequality follows from upper
semicontinuity. Thus these dimensions are locally constant. By
Grauert's theorem \cite{Grauert}, $p_*\mathcal O_\X(mA)$ is locally
free with fiber $H^0(X_t,mA_t)$ at $t$, and it commutes with
holomorphic base change, as in
Proposition~\ref{common:weighted-extension}. Taking $M=0$, we obtain
Theorem~\ref{thm:main}.
\end{proof}

\section{Applications and boundary variants}
\label{sec:applications}
In this section we first pass from disks to arbitrary smooth bases,
and show that over Stein bases every section on a fiber extends
globally. We then consider semipositive parts of divisors, rounded
boundaries, and projectively flat coefficient bundles. In the last
subsection we consider boundaries with coefficients equal to one;
here vanishing along the boundary, or a cohomological hypothesis for
lifting, reduces the problem to the klt case. In each application we
specify the line bundle to which the extension theorem is applied.

\subsection{Base change and global extension}

The following theorem is a consequence of the extension theorem over
disks, Grauert's theorem \cite{Grauert} and Cartan's Theorem~B. We
give the reduction in order to explain the assertions over
higher-dimensional bases and over Stein bases.

\begin{theorem}
\label{thm:smooth-base}
Let $f:\X\to S$ be a proper holomorphic submersion with connected
fibers $X_s=f^{-1}(s)$, where $S$ is a complex manifold and $\X$ is
K\"ahler. We denote by $K_{\X/S}=K_\X\otimes f^*K_S^{-1}$ the relative
canonical bundle. Let $B\geq0$ be a rational relative SNC divisor
with coefficients strictly less than one. For every integer $m\ge1$
with $mB$ integral, put
\[
 L_m=\mathcal O_\X\bigl(m(K_{\X/S}+B)\bigr),
 \qquad \mathcal E_m=f_*L_m.
\]
Then $\mathcal E_m$ is locally free and commutes with arbitrary
holomorphic base change.
Every section on a fiber extends over the inverse image of a
neighborhood of its base point.
If $S$ is Stein, every section on a fiber extends over the whole of
$\X$. The same conclusions hold with a twist by a rational line
bundle $M$ with a smooth semipositive metric on $\X$, in the degrees
for which $mM$ is integral, with the convention on admissibility in
Section~\ref{sec:common-extension}.
\end{theorem}

We say that $\mathcal E_m$ commutes with base change if, for every
holomorphic map $\lambda:S'\to S$ of complex spaces, with fiber
product $\X_{S'}=\X\times_SS'$ and projections
$\lambda':\X_{S'}\to\X$ and $f':\X_{S'}\to S'$, the natural map
\[
 \lambda^*\mathcal E_m\longrightarrow f'_*\bigl((\lambda')^*L_m\bigr)
\]
is an isomorphism. In particular, taking $S'=\{s\}$, we identify the
fiber $\mathcal E_m(s)$ of $\mathcal E_m$ at $s$ with
$H^0(X_s,L_m|_{X_s})$. In the twisted case we put
$L_m=\mathcal O_\X(m(K_{\X/S}+B+M))$.

\begin{proof}
We work in a coordinate ball $U\subset S$. Any two points of $U$ lie
on a connected disk contained in $U$, namely the intersection of $U$
with the complex affine line through them. The inverse image of this
disk is a K\"ahler submanifold of $\X$, and the restricted boundary is
relative SNC. Applying Theorem~\ref{thm:main} at each point of the
disk, we see that the dimension of the space of sections on the
fibers is locally constant, and hence constant, along the disk.
Therefore it is constant on $U$. By Grauert's theorem \cite{Grauert},
$\mathcal E_m$ is locally free, and $\mathcal E_m(s)$ is identified
with $H^0(X_s,L_m|_{X_s})$ for every $s$. Since $f$ is a submersion,
the line bundle $L_m$ is flat over $S$. Therefore, by the base change
theorem for analytic cohomology
\cite[Chapter~III, Theorem~3.4(b),(d)]{BS}, these identifications
hold for arbitrary holomorphic base change. Finally, every vector in
the fiber of a locally free sheaf extends to a local section; this
gives the local extension of sections on $X_s$.

Suppose that $S$ is Stein, and let $s\in S$. Let $\mathcal I_s$ be
the ideal sheaf of $s$, and regard $\mathcal E_m(s)$ as a skyscraper
sheaf at $s$. From the exact sequence
\[
 0\longrightarrow\mathcal I_s\mathcal E_m
 \longrightarrow\mathcal E_m
 \longrightarrow \mathcal E_m(s)
 \longrightarrow0
\]
we obtain the following part of the exact cohomology sequence:
\[
 H^0(S,\mathcal E_m)\longrightarrow\mathcal E_m(s)
 \longrightarrow H^1(S,\mathcal I_s\mathcal E_m).
\]
Since $S$ is Stein and $\mathcal I_s\mathcal E_m$ is coherent, the
last group vanishes by Cartan's Theorem~B. Hence every vector in
$\mathcal E_m(s)$ is the value of a global section of $\mathcal E_m$.
Since $H^0(S,\mathcal E_m)=H^0(\X,L_m)$ by the definition of the
direct image, and $\mathcal E_m(s)=H^0(X_s,L_m|_{X_s})$ by base
change, the restriction map
\[
 H^0(\X,L_m)\longrightarrow H^0(X_s,L_m|_{X_s})
\]
is surjective. Hence every section on $X_s$ extends over the whole
of $\X$. The same argument applies to the twisted line bundle.
\end{proof}

\begin{remark}
It suffices to assume that $\X$ is K\"ahler locally over $S$, that
is, that every point of $S$ has a neighborhood $U$ such that
$f^{-1}(U)$ is K\"ahler. In fact, the proof uses the K\"ahler
property only over coordinate disks, and the assertion for Stein
bases uses only the local freeness of the direct image. In
particular, over a disk every section on a fiber extends over the
whole family in each admissible degree, without any shrinking of the
disk depending on the degree.
\end{remark}

\subsection{Semipositive parts and rounded boundaries}

A divisor may contribute to the adjoint bundle either through its
singular weight or through a smooth semipositive metric on its
rational line bundle. The bound on the coefficients is needed for the
integrability of the singular weight. The following corollary
combines these two possibilities.

\begin{corollary}\label{cor:split}
Suppose that a rational divisor \(C\) satisfies
\(\mathcal O_\X(C)\simeq_\Q\mathcal O_\X(B)\otimes M\), where
$\simeq_\Q$ denotes an isomorphism after a common positive integral
tensor power, and where \(B\) and \(M\) satisfy the hypotheses of
Theorem~\ref{thm:smooth-twist}.
Then the sections of \(m(K_{X_0}+C_0)\) extend in every degree
in which these data and the isomorphism are integral.
If \(K_{X_t}+C_t\) is nef on every fiber, the assertion follows also
from Theorem~\ref{thm:nef}.
\end{corollary}
\begin{proof}
In such a degree the given isomorphism identifies
$m(K_{\X/\Delta}+C)$ with $m(K_{\X/\Delta}+B+M)$. We apply
Theorem~\ref{thm:smooth-twist} to the latter bundle, or
Theorem~\ref{thm:nef} if its restrictions to the fibers are nef.
\end{proof}
For instance, if \(\mathcal O_\X(C)\) itself carries a smooth
semipositive metric, we can take \(B=0\), and then no bound on the
coefficients of \(C\) is needed. The condition $b_i<1$ in the main
theorems is the condition under which the divisor weight alone is
integrable against every section.

For real coefficients, rounding in a fixed degree gives a rational
boundary, to which the theorem applies directly; no limiting argument
in the coefficients is needed.

\begin{corollary}\label{cor:round-twist}
In the family setting of Theorem~\ref{thm:smooth-base}, let $M$
be a rational line bundle with a smooth semipositive metric on $\X$,
and let \(B=\sum b_iD_i\) be an effective real relative SNC divisor
with \(0\le b_i<1\). Then, for every positive integer \(m\) such that
\(mM\) is integral, every section of
\[
 \mathcal O_{X_s}(mK_{X_s}+\lfloor mB_s\rfloor+mM_s)
\]
extends locally. The dimension of the space of such sections is
locally constant, and the corresponding direct image is locally free
and commutes with base change.
\end{corollary}
\begin{proof}
For the fixed integer \(m\) we put
$B^{(m)}=m^{-1}\lfloor mB\rfloor$. This is an effective rational
relative SNC boundary with coefficients less than one, and
\[
 m(K_{\X/S}+B^{(m)}+M)
       =mK_{\X/S}+\lfloor mB\rfloor+mM.
\]
Hence Theorem~\ref{thm:smooth-twist} and then
Theorem~\ref{thm:smooth-base} apply.
\end{proof}

This reduction uses the general theorem. Since nefness is not
preserved when the boundary is lowered, Theorem~\ref{thm:nef} applies
in degree $m$ only if \(K_{X_s}+B_s^{(m)}+M_s\) is nef on every fiber.

\begin{remark}
\label{rem:poincare-obstruction}
For the twisted statements, nefness on the fibers, and even smooth
semipositivity on the fibers, is not sufficient. Let $C$ be an
elliptic curve, and choose a holomorphic map
$\Delta\to\operatorname{Pic}^0(C)$ which sends $0$, and no other
point, to the origin $[\mathcal O_C]$. Pulling back the Poincar\'e
bundle, we obtain a line bundle $M$ on $C\times\Delta$ with
$M_0\simeq\mathcal O_C$ and with $M_t$ nontrivial of degree zero for
$t\ne0$. Every $M_t$ carries a smooth flat metric, but
\[
 h^0(C,K_C\otimes M_0)=1,
 \qquad
 h^0(C,K_C\otimes M_t)=0\quad(t\ne0).
\]
Hence the central section does not extend. Thus, already for $B=0$
and $m=1$, the semipositivity of a metric on the total space cannot
be replaced by the semipositivity of metrics chosen separately on the
fibers.
\end{remark}

\subsection{Projectively flat coefficients}

We now apply Theorem~\ref{thm:smooth-twist} and its consequence for
base change, Theorem~\ref{thm:smooth-base}, to sections with
projectively flat coefficients. On the projective bundle of lines in
the dual bundle, these vector-valued sections correspond to sections
of a line bundle. The projective flatness gives the smooth
semipositive twist needed for the scalar theorem on this projective
bundle.

\begin{theorem}
\label{thm:projectively-flat}
In the setting of Theorem~\ref{thm:smooth-base}, let $M$ be a
rational line bundle with a smooth semipositive metric $h_M$, and
let $(E,h_E)$ be a Hermitian holomorphic vector bundle of rank
$r\ge1$ on $\X$ such that
\[
 \Theta(E,h_E)=\eta\,\mathrm{Id}_E,
 \qquad \eta\geq0,
\]
where $\Theta(E,h_E)$ is the Chern curvature and $\eta$ is a smooth
closed real $(1,1)$-form on $\X$. In particular, $E$ may be any
unitary flat vector bundle ($\eta=0$); we do not assume that its
monodromy is finite.
Then, for every integer $m\ge1$ with $mB$ and $mM$ integral, every
section of
\[
 \mathcal O_{X_s}\bigl(m(K_{X_s}+B_s+M_s)\bigr)\otimes E_s
\]
extends locally over $S$. The corresponding direct image is locally
free and commutes with holomorphic base change. If $S$ is Stein,
the extension exists over the whole of $\X$.
\end{theorem}
\begin{proof}
\emph{Step 1: The projective bundle.}
Fix a K\"ahler form $\omega_\X$ on $\X$. Let
$\pi:\mathbb P(E^*)\to\X$ be the projective bundle of lines in the
dual bundle $E^*$; its fiber over $x\in\X$ is the projective space
$\mathbb P(E_x^*)$ of one-dimensional subspaces of $E_x^*$. The
tautological line bundle
$\mathcal O_{\mathbb P(E^*)}(-1)\subset\pi^*E^*$ has the fiber $\ell$
at the point $[\ell]$, and its dual
$
 H=\mathcal O_{\mathbb P(E^*)}(1)
$
is called the \emph{hyperplane line bundle}. Dualizing the inclusion
$\mathcal O_{\mathbb P(E^*)}(-1)\subset\pi^*E^*$, we see that $H$ is a
quotient of $\pi^*E$. The global sections of $H$ on
$\mathbb P(E_x^*)$ are the linear forms on $E_x^*$, that is, the
elements of $E_x$. Hence
\[
 \pi_*H=E,
 \qquad
 K_{\mathbb P(E^*)/S}=\pi^*K_{\X/S}-rH+\pi^*\det E.
\]
The second formula follows from the relative Euler sequence
\[
 0\longrightarrow\mathcal O_{\mathbb P(E^*)}
 \longrightarrow\pi^*E^*\otimes H
 \longrightarrow T_{\mathbb P(E^*)/\X}
 \longrightarrow0,
\]
which gives $K_{\mathbb P(E^*)/\X}=-rH+\pi^*\det E$, and from
$K_{\mathbb P(E^*)/S}=K_{\mathbb P(E^*)/\X}+\pi^*K_{\X/S}$. We denote
by $\Theta(H)$ and $\Theta(\det E)$ the curvatures of the quotient
metric and of the determinant metric induced by $h_E$. Then
\[
 \Theta(H)=\omega_{\mathrm{FS}}+\pi^*\eta,
 \qquad
 \Theta(\det E)=r\eta,
\]
where $\omega_{\mathrm{FS}}:=\Theta(H)-\pi^*\eta$ is the relative
Fubini--Study form. This is a globally defined closed form.
In local trivializations adapted to the projectively flat connection,
it is semipositive, and positive in the fiber directions of $\pi$.
Consequently, $\pi^*\omega_\X+\omega_{\mathrm{FS}}$ is a K\"ahler
form on $\mathbb P(E^*)$.

\medskip\noindent
\emph{Step 2: The scalar bundle as a twisted log-adjoint bundle.}
For the fixed degree $m$ we define the rational line bundle
\[
 N_m=\pi^*M+\left(r+\frac1m\right)H-\pi^*\det E.
\]
The induced metric on $N_m$ is smooth and semipositive, since its
curvature is
\[
 \pi^*\Theta(h_M)
 +\left(r+\frac1m\right)\omega_{\mathrm{FS}}
 +\frac1m\pi^*\eta\geq0.
\]
Moreover, $mN_m=\pi^*(mM)+(mr+1)H-m\pi^*\det E$ is an integral line
bundle. The boundary $\pi^*B$ is relative SNC and klt, and
\begin{equation}\label{eq:projective-bundle-identity}
 m\bigl(K_{\mathbb P(E^*)/S}+\pi^*B+N_m\bigr)
 =\pi^*\bigl(m(K_{\X/S}+B+M)\bigr)+H.
\end{equation}

\medskip\noindent
\emph{Step 3: Scalar extension.}
Put $L_m=\mathcal O_\X(m(K_{\X/S}+B+M))$ and
$\widetilde L_m=\pi^*L_m\otimes H$. By the projection formula and
$\pi_*H=E$, we have
$\pi_*\widetilde L_m\simeq L_m\otimes E$,
and the same identification holds on each fiber. Thus a given
vector-valued section of $L_m|_{X_s}\otimes E_s$ corresponds to a
section of the line bundle $\widetilde L_m|_{\mathbb P(E_s^*)}$.
By \eqref{eq:projective-bundle-identity}, $\widetilde L_m$ has the
form required in the scalar extension theorem, with boundary
$\pi^*B$ and smooth semipositive twist $N_m$. Applying
Theorems~\ref{thm:smooth-twist} and~\ref{thm:smooth-base} to
$f\circ\pi$, we extend this section. Under the same identification,
the extension gives the required vector-valued section on $\X$.
If $S$ is Stein, the scalar extension is global, and hence so is the
vector-valued extension.

By the identification
$(f\circ\pi)_*\widetilde L_m\simeq f_*(L_m\otimes E)$, the local
freeness and the compatibility with base change also follow from the
scalar theorem. The construction is made separately for each
admissible $m$; hence the twist $N_m$ may depend on $m$.
\end{proof}

\begin{remark}
If we replace $1/m$ by $k/m$ for an integer $k\geq0$, then $H$ is
replaced by $kH$ on the right-hand side of
\eqref{eq:projective-bundle-identity}. Since
$\pi_*\mathcal O_{\mathbb P(E^*)}(k)=\operatorname{Sym}^kE$, the same
argument applies to symmetric powers, with twist curvature
$\pi^*\Theta(h_M)+(r+k/m)\omega_{\mathrm{FS}}+(k/m)\pi^*\eta$.
Put $A=K_{\X/S}+B+M$, and let
$\pi_s:\mathbb P(E_s^*)\to X_s$ and $H_s=H|_{\mathbb P(E_s^*)}$
be the projective bundle and the hyperplane line bundle on the
fiber. If \(A_s\) is nef on every fiber, then the restriction to
$\mathbb P(E_s^*)$ of the log-adjoint class in
\eqref{eq:projective-bundle-identity} is
\(\pi_s^*A_s+H_s/m\) (respectively \(\pi_s^*A_s+kH_s/m\)), which is
nef, since \(H\) carries a smooth semipositive metric. In this case
the reduction uses only Theorem~\ref{thm:nef}, and the minimal model
program is not used.
\end{remark}

\subsection{Coefficients equal to one}

If a boundary coefficient is equal to one, the corresponding divisor
weight is not locally integrable. A section which vanishes along all
such components can be divided by the canonical section of their
reduced sum. In degree $m$ this lowers each of these coefficients by
$1/m$. The resulting boundary is klt, and Theorem~\ref{thm:main}
applies. For sections without this vanishing we use instead
adjunction and an additional cohomological hypothesis for lifting.

\begin{proposition}
\label{prop:lc-vanishing}
Let $p:\X\to\Delta$ be a proper holomorphic submersion with
connected compact fibers and K\"ahler total space.
Let $B=C+D$ be an effective rational relative SNC divisor with
coefficients at most one, where $D=\lfloor B\rfloor$ is the reduced
sum of the components
with coefficient one and $C=B-D$ has coefficients strictly less
than one and no common components with $D$. Let $m\ge1$ be an integer
with $mB$ integral. Then every section of $m(K_{X_0}+B_0)$ which
vanishes along the reduced divisor $D_0$, to order at least one,
extends locally. The direct image of
$\mathcal O_\X(m(K_{\X/\Delta}+B)-D)$ is locally free and commutes
with base change.
\end{proposition}
\begin{proof}
The rational boundary
\[
 B^{(m)}=C+\left(1-\frac1m\right)D
\]
is effective, relative SNC and klt, and $mB^{(m)}=mB-D$ is integral.
Since the given section vanishes along $D_0$, dividing it by the
canonical section of $D_0$, we obtain a holomorphic section in degree
$m$ for the boundary $B^{(m)}$. We apply the klt theorem to
$m(K_{\X/\Delta}+B^{(m)})$, and multiply the extension by the
canonical section of $D$. The assertion on the direct image follows
from the same theorem for $B^{(m)}$.
\end{proof}

\begin{remark}
The argument extends to an arbitrary effective real relative SNC
divisor $B=\sum b_iD_i$, without any upper bound on its coefficients.
Put
\[
 L_m=mK_{\X/\Delta}+\lfloor mB\rfloor,
 \qquad
 Q_m=\sum_i\max\{0,\lfloor mb_i\rfloor-m+1\}D_i.
\]
Then every section of $L_m|_{X_t}$ vanishing along $Q_m|_{X_t}$
extends, since
\[
 L_m-Q_m=m(K_{\X/\Delta}+B_m),
 \qquad
 B_m=\frac1m\sum_i\min\{\lfloor mb_i\rfloor,m-1\}D_i,
\]
and $B_m$ is an effective klt boundary. This concerns only the
indicated subspace of sections; it does not give the extension of all
sections when some coefficient is at least one.
\end{remark}

If $D\to\Delta$ is smooth, then, by adjunction, the restriction of a
general section to $D$ is identified with a klt pluricanonical
section in one lower dimension. This restriction extends by
Theorem~\ref{thm:main}, and, under a cohomological vanishing
hypothesis, we can lift the extension to the total space. The
assertions on the direct images then follow from Grauert's theorem
\cite{Grauert}.

\begin{proposition}
\label{prop:plt-residue}
Let $p:\X\to\Delta$ and $B=D+C$ satisfy the geometric hypotheses of
Proposition~\ref{prop:lc-vanishing}, and assume in addition that
$D\to\Delta$ is smooth; $D$ may have several disjoint components. For
a fixed integer $m\ge1$ with $mB$ integral, put
\[
 L=\mathcal O_\X(m(K_{\X/\Delta}+B)),
 \qquad F=L\otimes\mathcal O_\X(-D).
\]
If $H^1(X_0,F|_{X_0})=0$, then every section of $L|_{X_0}$ extends.
After shrinking the disk, $p_*L$ is locally free and commutes with
base change.
\end{proposition}
\begin{proof}
\emph{Step 1: Extension on $D$.}
Let $i:D\hookrightarrow\X$ be the inclusion, and let
$q_D=p|_D:D\to\Delta$ be the restricted family. By upper
semicontinuity, $H^1(X_t,F|_{X_t})=0$ on a smaller disk, and hence
$R^1p_*F=0$ there by Grauert's theorem. By adjunction we have
\[
 L|_D=\mathcal O_D\bigl(m(K_{D/\Delta}+C|_D)\bigr),
\]
and the boundary $C|_D$ is relative SNC and klt. Under this
adjunction isomorphism, the restriction to $D$ is the residue map
used here. Applying Theorem~\ref{thm:main} in dimension $n-1$ to each
connected component of $D$ over a smaller disk, we obtain the
required extension on $D$: for a given $s_0\in H^0(X_0,L|_{X_0})$,
the restriction of $s_0$ to $D_0$ extends to a section $r$ of
$L|_D$.

\medskip\noindent
\emph{Step 2: Lifting and correcting the central restriction.}
Choose a smaller Stein disk $U$. Since $R^1p_*F=0$, we have
$H^1(p^{-1}(U),F)=0$ by the Leray spectral sequence and Cartan's
Theorem~B. Hence, by the exact sequence
\[
 0\longrightarrow F\longrightarrow L\longrightarrow i_*(L|_D)
 \longrightarrow0
\]
the section $r$ lifts to a section $s'$ of $L$ over $p^{-1}(U)$.
This lift coincides with $s_0$ on $D_0$, but it need not coincide
with $s_0$ on the whole of $X_0$. The difference
$\delta_0=s_0-s'|_{X_0}$ vanishes along $D_0$, and hence it is a
section of $F|_{X_0}$.

Let $j_0:X_0\hookrightarrow p^{-1}(U)$ be the inclusion of the fiber,
and denote again by $t$ the pullback of the coordinate of the disk.
Since $X_0$ is defined by $t=0$, restriction to $X_0$ gives the exact
sequence
\[
 0\longrightarrow F\xrightarrow{\,t\,}F
 \longrightarrow (j_0)_*(F|_{X_0})\longrightarrow0.
\]
By the same vanishing of $H^1$, $\delta_0$ extends to a section
$\delta$ of $F$. Then, under the inclusion $F\hookrightarrow L$, the
section $s'+\delta$ restricts to $s_0$ on $X_0$.

\medskip\noindent
\emph{Step 3: Direct images.}
The argument for $F$ applies at every nearby fiber. Consequently, by
the extension of bases, upper semicontinuity and Grauert's theorem,
$p_*F$ is locally free and commutes with base change; by the theorem
in lower dimension, the same holds for $(q_D)_*(L|_D)$. Hence the
sequence of direct images
\[
 0\longrightarrow p_*F\longrightarrow p_*L
 \longrightarrow (q_D)_*(L|_D)\longrightarrow0
\]
has locally free end terms. Therefore it splits locally, and $p_*L$
is locally free. On every fiber the corresponding sequence of spaces
of sections is exact, since $H^1(X_t,F|_{X_t})=0$. Comparing these
sequences, we obtain the base change isomorphism for $p_*L$ from
those for the end terms. We note that the klt theorem is used only
on $D$, in dimension $n-1$.
\end{proof}

\begin{corollary}\label{cor:surface-residue}
In Proposition~\ref{prop:plt-residue}, suppose that the fibers $X_t$
are surfaces. Then the conclusion follows from the vanishing of $H^1$
and the Riemann--Roch theorem for curves, without any extension
theorem in higher dimension and without the minimal model program.
\end{corollary}
\begin{proof}
It suffices to prove the local extension for the smooth family of
curves $D\to\Delta$. Let $Y$ be a connected component of a fiber of
$D$, of genus $g$, and let $\Gamma=C|_Y$; this is a rational
effective klt boundary. For $m\Gamma$ integral, both $g$ and the
degree of $m(K_Y+\Gamma)$ are constant in the family. We show that
the dimension of the space of sections of $m(K_Y+\Gamma)$ is
determined by these data.

If $g=0$, the space of sections has dimension
$\max\{0,\deg(m(K_Y+\Gamma))+1\}$. If $g=1$ and the degree is
positive, the dimension is equal to the degree; if the degree is
zero, then $\Gamma=0$ and $mK_Y$ is trivial.

If $g\ge2$ and $m\ge2$, the degree exceeds $2g-2$; hence, by the
Riemann--Roch theorem, $h^0=\deg(m(K_Y+\Gamma))+1-g$. If $m=1$, then
$\Gamma=0$ by integrality and the klt condition, and
$h^0(K_Y)=g$.

In every case the dimension is locally constant. By Grauert's
theorem we therefore have the local extension on $D$, and
Proposition~\ref{prop:plt-residue} applies. If $D$ is disconnected,
each component is treated separately.
\end{proof}

\noindent\textbf{Acknowledgments.}
The first named author is partially supported by the National Key Research and Development Program of China (Grant No.~2024YFA1013201), the Natural Science Foundation of Chongqing (Grant No.~CSTB2024NSCQ-LZX0040), and the Special Project of Chongqing Municipal Science and Technology Bureau (Grant No.~2025CCZ015). The second named author is sponsored by the National Key R\&D Program of China (Grant No.~2024YFA1013200) and the National Natural Science Foundation of China (Grant No.~12671100).

\noindent\textbf{AI Declaration}: The authors used ChatGPT as an auxiliary tool in this work. The authors verified and completed all mathematical arguments and take full responsibility for the content of this paper.

\end{document}